\documentclass[a4paper,11pt]{amsart}

\usepackage[T1]{fontenc}
\usepackage[utf8]{inputenc}
\usepackage{lmodern}
\usepackage[english]{babel}

\usepackage[pdfusetitle,linktocpage,colorlinks=false,hidelinks]{hyperref}

\usepackage{amsmath,amssymb,amsthm,amscd,mathtools}

\newtheorem{thm}{Theorem}[section]
\newtheorem{prop}[thm]{Proposition}
\newtheorem{lem}[thm]{Lemma}
\newtheorem{cor}[thm]{Corollary}

\newtheorem{defn}[thm]{Definition}
\theoremstyle{remark}
\newtheorem{rem}[thm]{Remark}

\newcommand{\C}{\mathbb C}
\newcommand{\N}{\mathbb N}

\newcommand{\ran}{\operatorname{ran}}
\newcommand{\rank}{\operatorname{rank}}
\newcommand{\Prop}{\operatorname{prop}}

\newcommand{\supp}{\operatorname{supp}}
\newcommand{\spann}{\operatorname{span}}
\newcommand{\clspan}[1]{\overline{\operatorname{span}}\, #1}
\newcommand{\Bup}{B_u^p(X)}
\newcommand{\UBp}{UB^p(X)}

\newcommand{\Ind}{\operatorname{Ind}}
\newcommand{\insupp}{\operatorname{insupp}}
\newcommand{\dcore}{\operatorname{dcore}}

\begin{document}

\title{Morita Induction and Preservation of Geometric and Dynamical Ideals}
\author{Yeong Chyuan Chung}
\address{School of Mathematics, Jilin University, Changchun 130012, Jilin, P. R. China}
\email{chungyc@jlu.edu.cn}
\author{Xinhui Du}
\address{School of Mathematics, Jilin University, Changchun 130012, Jilin, P. R. China}
\email{xhdu24@mails.jlu.edu.cn}
\date{\today}


\begin{abstract}
Given Banach algebras with bounded approximate identities that are Morita equivalent in the sense of Paravicini, we give explicit formulas for the induced correspondence of closed ideals, characterize it via localized Morita submodules, determine when corresponding ideals are Morita equivalent via restricted bimodules, and show that corresponding quotients are always Morita equivalent. We prove preservation of geometric cores associated to compatible dense algebraic cores. For equivalent étale groupoids with paracompact unit spaces, Morita induction preserves dynamical ideals and inner supports. For the Morita equivalence between the \(\ell^p\) uniform Roe algebra and the \(\ell^p\) uniform algebra, we prove that the latter has a bounded approximate identity of projections, identify its geometric ideals explicitly, and identify corresponding complete-lattice block decompositions of the full ideal lattices with matching extremal ideals. Under property A, or for \(p=1\), these blocks collapse; for \(p=2\), collapse of the minimal block characterizes property A.
\end{abstract}

\maketitle
\tableofcontents

\section{Introduction}

Morita equivalence is a general mechanism for transporting structural information between algebras.  For $C^*$-algebras, Rieffel's theory of strong Morita equivalence \cite{Rieffel1974Induced,Rieffel1974Morita} has become an indispensable tool in operator algebras.  For Banach algebras, Gr{\o}nb{\ae}k developed a categorical Morita theory for algebras with bounded approximate identities \cite{Gronbaek1995}.  A Banach pair formulation suited to $KK^{\mathrm{ban}}$-theory was developed systematically by Paravicini, together with its relation to Banach algebra $K$-theory \cite{Paravicini2009Morita}; see also \cite{Paravicini2009Induction,Paravicini2015}.

In a previous paper \cite{Chung2025}, the first author proved that, for a metric space $X$ with bounded geometry and $1\leq p<\infty$, the $\ell^p$ uniform Roe algebra $B_u^p(X)$ and the $\ell^p$ uniform algebra $UB^p(X)$ are Morita equivalent in the sense of Paravicini.  This
motivates the present study of ideal induction for Morita equivalent
Banach algebras.  The existence of an ideal correspondence is natural
from general Morita theory.  Our emphasis is instead on concrete formulas
in the Banach pair language and on criteria that allow additional
structure on distinguished classes of ideals to be transported.  This is
particularly relevant for geometric ideals in Roe-type algebras, whose
definition involves a specified dense algebraic core and is therefore
not, by itself, an abstract Morita-invariant notion.

We first develop the required ideal induction in an abstract Banach
algebra setting.  For a closed ideal $I\triangleleft B$, we identify the
closed submodules cut out by $I$ and show that they are precisely the norm
closures of the natural module products with $I$.  This gives explicit
formulas for the induced ideal $\Ind_B^A(I)\triangleleft A$ and leads to a direct proof
that Morita induction is an order isomorphism between the closed ideal
lattices of $A$ and $B$ (Theorem~\ref{thm:ideal-lattice}).  If $I\triangleleft B$ corresponds to
$J\triangleleft A$, the closed linear spans of the restricted pairing
ranges are exactly $I^2$ and $J^2$.  Consequently, in the nondegenerate
Morita theory used throughout the paper, the canonical restricted
bimodules implement a Morita equivalence between the corresponding
ideals exactly when they are square-dense.  We also note that
Paravicini's later definition for possibly degenerate Banach algebras
removes this restriction \cite{Paravicini2015}.  Finally, corresponding
quotient algebras are always Morita equivalent.
The restricted Morita context will also reappear in both applications:
corresponding dynamical ideals for equivalent groupoids, and
corresponding coarse-geometric ideals in the Roe-type setting, are
themselves Morita equivalent through the restricted bimodules.

The first mechanism for preserving additional ideal structure uses
dense algebraic cores.  If $A_0\subseteq A$ and $B_0\subseteq B$ are
distinguished dense subalgebras, we call a closed ideal
$J\triangleleft A$ $A_0$-geometric when $J\cap A_0$ is dense in $J$.
We formulate a notion of compatible dense Morita subcontext and prove
that, under this hypothesis,
\[
 I\text{ is }B_0\text{-geometric}
 \quad\Longleftrightarrow\quad
 \Ind_B^A(I)\text{ is }A_0\text{-geometric}.
\]
Compatibility of the Morita bimodules with the two
dense algebraic cores shows that geometricity ---  additional structure beyond the
abstract ideal lattice --- is preserved under induction (Theorem~\ref{thm:geometric-preservation}).  
More generally, for an arbitrary closed ideal, Morita induction
transports its geometric core, namely the largest geometric ideal
contained in it (Proposition~\ref{prop:geometric-core-preservation}). Once the appropriate ideal induction formulas
are available, the proof is direct, and the formulation separates the Morita-theoretic
argument from the special features of Roe-type algebras.

A second preservation mechanism is useful when ideals are more
naturally detected by their action on the Morita bimodule.  Writing
${}_AE_B^{>}$ for the $A$--$B$ bimodule in the Morita equivalence,
we prove that closed ideals $J\triangleleft A$ and $I\triangleleft B$
correspond under Morita induction if and only if
\[
 \overline{JE^{>}}=\overline{E^{>}I}.
\]
Thus the corresponding ideal is completely determined by the closed
localized submodule it generates.  We use this characterization for
equivalent groupoids.  The algebraic localization identity on the
equivalence space identifies the localized submodules associated with
corresponding invariant open subsets, showing first that dynamical
ideals correspond and are themselves Morita equivalent.  More
generally, Morita induction transports the inner support of every
closed ideal under the canonical lattice isomorphism between invariant
open subsets of the two unit spaces.

We then apply the dense core mechanism to the Morita equivalence between
$B_u^p(X)$ and $UB^p(X)$.  One technical issue must first be addressed.
The ideal induction theory used here assumes bounded approximate
identities.  The algebra $B_u^p(X)$ is unital, whereas $UB^p(X)$ is
usually nonunital.  We prove that $UB^p(X)$ has a bounded two-sided
approximate identity consisting of projections (Theorem~\ref{thm:UBp-bai-projections}), thereby placing this Morita equivalence within the framework of Section~\ref{sect:ideal}.

The key technical input in this construction is a quantitative
Banach-space projection argument based on the uniform projection
property of $\ell^p$, due to Pe\l czy\'nski and Rosenthal
\cite{PelczynskiRosenthal1975}; see also
\cite[Section~9]{Heinrich1980}.  A quantitative form of the
construction in \cite[Proposition~1.1]{Johnson1980} produces a
uniformly bounded finite-rank projection which simultaneously fixes
prescribed finite-dimensional subspaces of $\ell^p$ and of its dual,
with a rank bound depending only on their dimensions.  This is
precisely the uniform control needed to construct projection local
units in $U\C^p[X]$.  Compare \cite[Lemma~4.2]{Lusky2003} for the
corresponding two-sided absorption result without the quantitative
rank bound.

Finally, the algebraic finite-propagation bimodules in
\cite{Chung2025} form a compatible dense subcontext for
\[
 A_0=U\C^p[X]\subseteq UB^p(X),
 \qquad
 B_0=\C_u[X]\subseteq B_u^p(X).
\]
Morita induction therefore identifies the full ideal lattices of
$B_u^p(X)$ and $UB^p(X)$ and restricts to an isomorphism between
their geometric ideal lattices.  We sharpen this by showing that, for
every ideal $\mathcal J$ of the bounded coarse structure,
\[
 \Ind_{B_u^p(X)}^{UB^p(X)}
 \bigl(B_u^p(X,\mathcal J)\bigr)
 =
 UB^p(X,\mathcal J).
\]
The ideals on the two sides are themselves Morita equivalent, and
$UB^p(X,\mathcal J)$ admits a bounded approximate identity of
projections.  Combining these results with the classification in our
companion paper \cite[Theorem~4.21(i) and Proposition~4.24]{ChungDu2026} gives an explicit coarse-geometric
classification of the geometric ideals of $UB^p(X)$.  
More generally, Morita induction identifies the decompositions of
the full ideal lattices into blocks determined by essential coarse
support and transports both the smallest and the largest ideals in
each corresponding block.
In particular, if $X$ has property~A, or if $p=1$, every closed ideal of $UB^p(X)$
is geometric and the entire ideal lattice is described by the
corresponding coarse-ideal lattice.
For $p=2$, conversely, property~A is characterized by every closed
ideal of $UB^2(X)$ being geometric, or equivalently by collapse of
the block associated with the minimal coarse ideal.

The paper is organized as follows.
Section~2 develops ideal induction for Morita equivalent Banach
algebras with bounded approximate identities.
Section~3 establishes preservation of ideals defined relative to
compatible dense algebraic cores and shows that Morita induction
preserves geometric cores.
Section~4 characterizes corresponding ideals through localized
Morita submodules.
Section~5 applies this characterization to dynamical ideals and inner
supports for reduced groupoid $L^p$-operator algebras.
Section~6 constructs bounded approximate identities of projections
for $\ell^p$ uniform algebras.
Section~7 applies the abstract results to the Morita equivalence
between $B_u^p(X)$ and $UB^p(X)$, establishes an explicit
coarse-ideal parametrization of geometric ideals, and then
identifies the corresponding essential-support and geometric-core
blocks, including their extremal ideals.

\section{Morita equivalence and induction of ideals} \label{sect:ideal}

Throughout the paper, all Banach spaces and Banach algebras are complex.  All Banach module actions are assumed to be contractive. Whenever two linear subspaces are multiplied using an algebra product or module action, juxtaposition denotes their algebraic linear span. Thus, for example,
\[
 E^{>}I=\spann\{\xi i:\xi\in E^{>},\ i\in I\},
 \qquad
 JE^{>}=\spann\{j\xi:j\in J,\ \xi\in E^{>}\}.
\]
Accordingly, $\overline{E^{>}I}$ and $\overline{JE^{>}}$ denote the
corresponding norm-closed linear spans. We use Paravicini's Banach pair formulation of Morita equivalence; see \cite{Paravicini2009Morita,Chung2025}.  We recall only the structure needed below.

\subsection{Morita contexts and bounded approximate identities}

\begin{defn}\label{def:morita}
Let $A$ and $B$ be Banach algebras.  A \emph{Morita equivalence} between $A$ and $B$ consists of Banach bimodules
\[
 {}_B E_A^{<},\qquad {}_A E_B^{>},
\]
and continuous bilinear pairings
\[
 \langle\cdot,\cdot\rangle_B:E^{<}\times E^{>}\longrightarrow B,
 \qquad
 {}_A\langle\cdot,\cdot\rangle:E^{>}\times E^{<}\longrightarrow A,
\]
with norm at most one, satisfying the usual module covariance and balancing identities
\begin{align}
\langle be^{<},e^{>} \rangle_B = b\langle e^{<},e^{>} \rangle _B, &\quad \langle e^{<},e^{>}b \rangle_B = \langle e^{<},e^{>} \rangle_Bb, \notag\\
{}_A\langle ae^{>},e^{<} \rangle = a{}_A\langle e^{>},e^{<} \rangle, &\quad {}_A\langle e^{>},e^{<}a \rangle = {}_A\langle e^{>},e^{<} \rangle a, \notag\\
 \langle e^{<}a,e^{>}\rangle_B&=\langle e^{<},ae^{>}\rangle_B, \notag\\
 {}_A\langle e^{>}b,e^{<}\rangle&={}_A\langle e^{>},be^{<}\rangle \notag
\end{align}
for $e^{<}\in E^{<}$, $e^{>}\in E^{>}$, $a\in A$ and $b\in B$, together with
\begin{align*}
 \langle e^{<},e^{>}\rangle_B f^{<}
 &=e^{<}\,{}_A\langle e^{>},f^{<}\rangle,\\
 e^{>}\langle f^{<},f^{>}\rangle_B
 &={}_A\langle e^{>},f^{<}\rangle f^{>},
\end{align*}
for $e^{<},f^{<}\in E^{<}$ and $e^{>},f^{>}\in E^{>}$.  We also require both pairings to be full,
\[
 \clspan{\langle E^{<},E^{>}\rangle_B}=B,
 \qquad
 \clspan{{}_A\langle E^{>},E^{<}\rangle}=A,
\]
and all four module actions to be nondegenerate.  Here nondegeneracy means that the linear span of the corresponding module products is dense; for example, $\overline{AE^{>}}=E^{>}$ and $\overline{E^{>}B}=E^{>}$. 
\end{defn}

\begin{defn}
A net $(u_\lambda)$ in a Banach algebra $A$ is a \emph{bounded approximate identity} if
\[
 \sup_\lambda\|u_\lambda\|<\infty,
 \qquad
 u_\lambda a\to a,
 \qquad
 au_\lambda\to a
\]
for every $a\in A$.
\end{defn}

We use the following standard observation repeatedly.

\begin{lem}\label{lem:bai-module}
Let $A$ have a bounded approximate identity $(u_\lambda)$.
\begin{enumerate}
\item If $E$ is a nondegenerate left Banach $A$-module, then $u_\lambda e\to e$ for every $e\in E$.
\item If $F$ is a nondegenerate right Banach $A$-module, then $f u_\lambda\to f$ for every $f\in F$.
\end{enumerate}
\end{lem}

\begin{proof}
We prove (i); the other case is analogous.  The assertion is immediate on the dense subspace $AE$. Indeed, if
\[
 e=\sum_{k=1}^n a_k f_k,
\]
then
\[
 \|u_\lambda e-e\|
 \leq
 \sum_{k=1}^n
 \|(u_\lambda a_k-a_k)f_k\|
 \leq
 \sum_{k=1}^n
 \|u_\lambda a_k-a_k\|\,\|f_k\|
 \longrightarrow0.
\]
The operators $e\mapsto u_\lambda e$ are uniformly bounded, so the convergence extends from the dense subspace $AE$ to all of $E$.
\end{proof}

\subsection{Submodules associated to an ideal}

For the rest of this section, assume that $A$ and $B$ have bounded approximate identities and that $(E^{<},E^{>})$ is a Morita equivalence between them.  We write $\mathcal I(A)$ for the complete lattice of closed two-sided ideals of a Banach algebra $A$, ordered by inclusion.

Let $I\triangleleft B$ be a closed two-sided ideal.  Define
\begin{align*}
 E_I^{>}
 &:={\bigl\{\xi\in E^{>}:\langle\eta,\xi\rangle_B\in I
 \text{ for all }\eta\in E^{<}\bigr\}},\\
 {}_I E^{<}
 &:={\bigl\{\eta\in E^{<}:\langle\eta,\xi\rangle_B\in I
 \text{ for all }\xi\in E^{>}\bigr\}}.
\end{align*}
These are closed $A$--$B$ and $B$--$A$ submodules, respectively.

\begin{prop}\label{prop:ideal-submodules}
For every closed ideal $I\triangleleft B$,
\[
 E_I^{>} = \overline{E^{>}I},
 \qquad
 {}_I E^{<}=\overline{I E^{<}}.
\]
\end{prop}

\begin{proof}
If $\xi\in E^{>}$ and $b\in I$, then
\[
 \langle\eta,\xi b\rangle_B=\langle\eta,\xi\rangle_Bb\in I,
\]
so $E^{>}I\subseteq E_I^{>}$.  Hence $\overline{E^{>}I}\subseteq E_I^{>}$.

For the reverse inclusion, let $\xi\in E_I^{>}$.  Let $(u_\lambda)$ be a bounded approximate identity for $A$.  By Lemma~\ref{lem:bai-module}, $u_\lambda\xi\to\xi$.  By fullness of the $A$-valued pairing, each $u_\lambda$ can be approximated in norm by a finite sum
\[
 a=\sum_{k=1}^n {}_A\langle\xi_k,\eta_k\rangle,
 \qquad \xi_k\in E^{>},\ \eta_k\in E^{<}.
\]
Then
\[
 a\xi
 =\sum_{k=1}^n {}_A\langle\xi_k,\eta_k\rangle\xi
 =\sum_{k=1}^n \xi_k\langle\eta_k,\xi\rangle_B.
\]
Since $\xi\in E_I^{>}$, each coefficient $\langle\eta_k,\xi\rangle_B$ belongs to $I$.  Thus $a\xi\in E^{>}I$.  Approximating $u_\lambda$ sufficiently well and then letting $\lambda$ tend to infinity gives $\xi\in\overline{E^{>}I}$.

The proof of ${}_I E^{<}=\overline{I E^{<}}$ is symmetric.  Indeed, if $\eta\in{}_I E^{<}$, then $\eta u_\lambda\to\eta$, and for
$a=\sum_k{}_A\langle\xi_k,\eta_k\rangle$ one has
\[
 \eta a
 =\sum_k \eta\,{}_A\langle\xi_k,\eta_k\rangle
 =\sum_k \langle\eta,\xi_k\rangle_B\eta_k\in I E^{<}.
\]
\end{proof}

The original ideal can be recovered from either of these submodules.

\begin{prop}\label{prop:recover-I}
For every closed ideal $I\triangleleft B$,
\[
 I=\clspan{\langle E^{<},E_I^{>}\rangle_B}
  =\clspan{\langle{}_I E^{<},E^{>}\rangle_B}.
\]
\end{prop}

\begin{proof}
Both right-hand sides are contained in $I$ by definition.  Let $(v_\lambda)$ be a bounded approximate identity for $B$ and let $b\in I$.  Since $v_\lambda b\to b$ and the $B$-valued pairing is full, we may approximate $v_\lambda$ by finite sums
\[
 c=\sum_{k=1}^n\langle\eta_k,\xi_k\rangle_B.
\]
Then
\[
 cb=\sum_{k=1}^n\langle\eta_k,\xi_kb\rangle_B,
\]
and $\xi_kb\in E^{>}I\subseteq E_I^{>}$.  It follows that $b$ belongs to the first closed span.  The second equality follows in the same way from $bv_\lambda\to b$, using
\[
 b\langle\eta_k,\xi_k\rangle_B
 =\langle b\eta_k,\xi_k\rangle_B,
 \qquad b\eta_k\in I E^{<}\subseteq{}_I E^{<}.
\]
\end{proof}

\subsection{Induced ideals and the lattice correspondence}

\begin{defn}\label{def:induction}
For a closed ideal $I\triangleleft B$, define
\[
 \Ind_B^A(I)
 :=\clspan{{}_A\langle E_I^{>},E^{<}\rangle}.
\]
\end{defn}

\begin{prop}\label{prop:induction-formulas}
Let $I\triangleleft B$ and set $J=\Ind_B^A(I)$.  Then
\begin{align}
 J
 &=\clspan{{}_A\langle E^{>}I,E^{<}\rangle}
 =\clspan{{}_A\langle E^{>},I E^{<}\rangle}
 =\clspan{{}_A\langle E^{>},{}_I E^{<}\rangle}.
 \label{eq:ind-formulas}\notag
\end{align}
In particular, $J$ is a closed two-sided ideal of $A$.
\end{prop}

\begin{proof}
The first equality follows from Proposition~\ref{prop:ideal-submodules} and continuity of the pairing.  By the $B$-balance identity,
\[
 {}_A\langle\xi b,\eta\rangle
 ={}_A\langle\xi,b\eta\rangle
 \qquad(\xi\in E^{>},\ \eta\in E^{<},\ b\in I),
\]
so the first two closed spans are equal.  Proposition~\ref{prop:ideal-submodules} gives the third.

The module covariance of the $A$-valued pairing shows that this common closed span is invariant under left and right multiplication by $A$.
\end{proof}

There is a useful module identity behind the inverse correspondence.

\begin{lem} \label{lem:induced-module}
Let $I\triangleleft B$ and $J=\Ind_B^A(I)$.  Then
\[
 \overline{J E^{>}}=E_I^{>},
 \qquad
 \overline{E^{<}J}={}_I E^{<}.
\]
\end{lem}

\begin{proof}
For $\xi\in E_I^{>}$, $\eta\in E^{<}$ and $\zeta\in E^{>}$, the Morita compatibility relation gives 
\[
 {}_A\langle\xi,\eta\rangle\zeta
 =\xi\langle\eta,\zeta\rangle_B\in E_I^{>}.
\]
Thus every element in the linear span of
${}_A\langle E_I^{>},E^{<}\rangle$ maps $E^{>}$ into
$E_I^{>}$.  Since $E_I^{>}$ is closed and the module action is
continuous, the same holds for its norm closure $J$.  Hence $ J E^{>}\subseteq E_I^{>}$, and so $\overline{JE^{>}}\subseteq E_I^{>}$.

Conversely, let $\xi\in E_I^{>}$.  Let $(v_\lambda)$ be a bounded approximate identity of $B$.  Since $E^{>}$ is a nondegenerate right $B$-module, $\xi v_\lambda\to\xi$.  Approximate $v_\lambda$ by a finite sum $\sum_k\langle\eta_k,\zeta_k\rangle_B$.  Then
\[
 \xi\sum_k\langle\eta_k,\zeta_k\rangle_B
 =\sum_k{}_A\langle\xi,\eta_k\rangle\zeta_k\in J E^{>}.
\]
Hence $E_I^{>}\subseteq\overline{J E^{>}}$.  The second identity is symmetric.
\end{proof}

By interchanging $A$ and $B$, and $E^{>}$ and $E^{<}$, Definition~\ref{def:induction} gives a map $\Ind_A^B$ from closed ideals of $A$ to closed ideals of $B$.

\begin{thm}[Morita induction of ideals]\label{thm:ideal-lattice}
Let $A$ and $B$ be Banach algebras with bounded approximate identities, and let $(E^{<},E^{>})$ be a Morita equivalence between them.  Then
\[
 \Ind_B^A:\mathcal I(B)\longrightarrow\mathcal I(A)
\]
is an order isomorphism, with inverse $\Ind_A^B$.  Hence it is an isomorphism of complete lattices of closed two-sided ideals.
\end{thm}

\begin{proof}
The map is order preserving by definition.  Let $I\triangleleft B$ and set $J=\Ind_B^A(I)$.  Applying the induction formula on the $A$-side and Lemma~\ref{lem:induced-module},
\[
 \Ind_A^B(J)
 =\clspan{\langle E^{<},J E^{>}\rangle_B}
 =\clspan{\langle E^{<},E_I^{>}\rangle_B}
 =I
\]
by Proposition~\ref{prop:recover-I}.  The symmetric argument gives
\[
 \Ind_B^A(\Ind_A^B(J))=J
\]
for every closed ideal $J\triangleleft A$.  Thus the two induction maps are mutually inverse order isomorphisms.  Since the closed ideals of a Banach algebra form a complete lattice, every order isomorphism preserves all existing infima and suprema.
\end{proof}

The same argument shows that
$I\mapsto E_I^{>}$ identifies $\mathcal I(B)$ order
isomorphically with the closed $A$--$B$ submodules of $E^{>}$;
analogously for $E^{<}$.

%

\subsection{Morita equivalence of corresponding ideals}

We next ask when the ideals corresponding under Theorem~\ref{thm:ideal-lattice} are themselves Morita equivalent.  A subtlety is that, under the convention of Definition~\ref{def:morita}, a Banach algebra occurring in a Morita equivalence must be square-dense.  The restricted Morita context makes this obstruction completely explicit.

For a closed ideal $K$ in a Banach algebra, write
\[
 K^2:=\overline{KK}
 =\overline{\spann\{k_1k_2:k_1,k_2\in K\}}.
\]
Thus $K$ is nondegenerate as a Banach algebra precisely when
$K^2=K$.

\begin{lem}\label{lem:restricted-pairing-squares}
Let $I\triangleleft B$ and set $J=\Ind_B^A(I)$.  Put
\[
 X:=E_I^{>},\qquad Y:={}_I E^{<}.
\]
Then the original pairings restrict to $I$- and $J$-valued pairings on $Y\times X$ and $X\times Y$, respectively, and
\[
 \clspan{\langle Y,X\rangle_B}=I^2,
 \qquad
 \clspan{{}_A\langle X,Y\rangle}=J^2.
\]
\end{lem}

\begin{proof}
The pairings take values in $I$ and $J$ by the definitions of $X$, $Y$ and $J$.  By Proposition~\ref{prop:ideal-submodules},
\[
 X=\overline{E^{>}I},\qquad Y=\overline{I E^{<}}.
\]
Hence, for $i,i'\in I$, $\eta\in E^{<}$ and $\xi\in E^{>}$,
\[
 \langle i\eta,\xi i'\rangle_B
 =i\langle\eta,\xi\rangle_Bi'\in I^2.
\]
Continuity of the pairing therefore gives
\[
 \clspan{\langle Y,X\rangle_B}\subseteq I^2.
\]
For the reverse inclusion, Proposition~\ref{prop:recover-I} gives
\[
 I=\clspan{\langle Y,E^{>}\rangle_B}.
\]
If $i_1,i_2\in I$, approximate $i_1$ by finite sums of terms
$\langle y_k,\xi_k\rangle_B$ with $y_k\in Y$ and $\xi_k\in E^{>}$.  Multiplication on the right by $i_2$ shows that $i_1i_2$ belongs to $\clspan{\langle Y,X\rangle_B}$ since $\xi_k i_2\in E^{>}I\subseteq X$.  Thus $I^2\subseteq
\clspan{\langle Y,X\rangle_B}$.

For the $A$-valued pairing, Lemma~\ref{lem:induced-module} gives
\[
 X=\overline{J E^{>}},\qquad Y=\overline{E^{<}J}.
\]
Thus, for $j,j'\in J$,
\[
 {}_A\langle j\xi,\eta j'\rangle
 =j\,{}_A\langle\xi,\eta\rangle j'\in J^2,
\]
so
\[
 \clspan{{}_A\langle X,Y\rangle}\subseteq J^2.
\]
On the other hand, Definition~\ref{def:induction} gives
\[
 J=\clspan{{}_A\langle X,E^{<}\rangle}.
\]
If $j_1,j_2\in J$, approximate $j_1$ by finite sums of terms
${}_A\langle x_k,\eta_k\rangle$ with $x_k\in X$ and $\eta_k\in E^{<}$.  Then multiplying on the right by $j_2$ shows that $j_1j_2\in\clspan{{}_A\langle X,Y\rangle}$ since $\eta_kj_2\in E^{<}J\subseteq Y$.  This proves the second equality.
\end{proof}

\begin{thm}\label{thm:corresponding-ideals-morita}
Let $I\triangleleft B$ and set $J=\Ind_B^A(I)$.  With
\[
 X=E_I^{>},\qquad Y={}_I E^{<},
\]
the following conditions are equivalent:
\begin{enumerate}
\item $I^2=I$;
\item $J^2=J$;
\item with the restricted actions and pairings, ${}_IY_J$ and ${}_JX_I$ form a Morita equivalence between $J$ and $I$ in the sense of Definition~\ref{def:morita}.
\end{enumerate}
In particular, if either $I$ or $J$ has a bounded approximate identity, then $I$ and $J$ are Morita equivalent via the restricted bimodules.
\end{thm}

\begin{proof}
Assume first that $I^2=I$.  By Proposition~\ref{prop:induction-formulas},
\[
 J=\clspan{{}_A\langle E^{>}I,E^{<}\rangle}.
\]
Since $I=\overline{II}$, the subspace $II$ is dense in $I$.
Hence the right-hand side is also the closed span of terms
${}_A\langle\xi i_1i_2,\eta\rangle$.  By the $B$-balance identity,
\[
 {}_A\langle\xi i_1i_2,\eta\rangle
 ={}_A\langle\xi i_1,i_2\eta\rangle,
\]
where $\xi i_1\in X$ and $i_2\eta\in Y$.  Hence Lemma~\ref{lem:restricted-pairing-squares} gives
\[
 J\subseteq\clspan{{}_A\langle X,Y\rangle}=J^2.
\]
The reverse inclusion is automatic, so $J^2=J$.  Applying the same argument to the inverse Morita equivalence and using Theorem~\ref{thm:ideal-lattice} proves the converse.  Thus (i) and (ii) are equivalent.

Assume these conditions.  Lemma~\ref{lem:restricted-pairing-squares} shows that the restricted pairings are full.  The restricted module actions are nondegenerate.  Indeed,
\begin{align*}
 \overline{XI}\supseteq\overline{(E^{>}I)I}
 &=\overline{E^{>}I^2}=X,
 \\
 \overline{IY}\supseteq\overline{I(IE^{<})}
 &=\overline{I^2E^{<}}=Y,
\end{align*}
and the reverse inclusions are clear.  Similarly, using Lemma~\ref{lem:induced-module} and $J^2=J$,
\[
 \overline{JX}=X,
 \qquad
 \overline{YJ}=Y.
\]
All module covariance, balancing and Morita compatibility identities are inherited from the original context.  Thus (iii) holds.

Conversely, if the restricted context in (iii) is a Morita equivalence, its pairings are full.  Lemma~\ref{lem:restricted-pairing-squares} then gives
$I=I^2$ and $J=J^2$.  This proves the equivalence.

Finally, a Banach algebra with a bounded approximate identity is square-dense.  Hence a bounded approximate identity for either $I$ or $J$ implies the equivalent conditions above.
\end{proof}

\begin{rem}\label{rem:degenerate-ideal-morita}
Paravicini later extended Morita equivalence to possibly degenerate Banach algebras by replacing fullness and nondegeneracy with a power condition; see \cite[Definition~1.4]{Paravicini2015}.  In his notation, one asks for integers $k,l\geq1$ such that the closed spans $J^k$ and $I^l$ are contained in the respective pairing ranges.  Lemma~\ref{lem:restricted-pairing-squares} shows that, for the restricted context above, the closed linear spans of the pairing ranges are exactly $J^2$ and $I^2$.  Thus Paravicini's condition is satisfied with $k=l=2$ for every pair of corresponding closed ideals, without assuming that either ideal is square-dense or has a bounded approximate identity.  Consequently, every pair of corresponding closed ideals is Morita equivalent in the sense of this later definition.  Throughout the remainder of the paper, however, we continue to use the nondegenerate convention of Definition~\ref{def:morita}.
\end{rem}

\subsection{Quotient algebras}

The ideal correspondence is compatible with passage to quotients.

\begin{thm}\label{thm:quotient-morita}
Let $I\triangleleft B$ and $J=\Ind_B^A(I)$.  Then $A/J$ and $B/I$ are Morita equivalent.  More precisely, the quotient bimodules
\[
 \overline E^{>}:=E^{>}/E_I^{>},
 \qquad
 \overline E^{<}:=E^{<}/{}_I E^{<}
\]
form a Morita equivalence between $A/J$ and $B/I$, with pairings induced from the original pairings.
\end{thm}

\begin{proof}
We first check that the module actions descend.  By Lemma~\ref{lem:induced-module},
\[
 J E^{>}\subseteq E_I^{>},
 \qquad
 E^{<}J\subseteq{}_I E^{<},
\]
and by Proposition~\ref{prop:ideal-submodules},
\[
 E^{>}I\subseteq E_I^{>},
 \qquad
 I E^{<}\subseteq{}_I E^{<}.
\]
Thus $\overline E^{>}$ is an $(A/J)$-$(B/I)$ bimodule and $\overline E^{<}$ is a $(B/I)$-$(A/J)$ bimodule.

Define
\begin{align*}
 \langle[\eta],[\xi]\rangle_{B/I}
 &:=\langle\eta,\xi\rangle_B+I,\\
 {}_{A/J}\langle[\xi],[\eta]\rangle
 &:={}_A\langle\xi,\eta\rangle+J.
\end{align*}
If $\xi$ is changed by an element of $E_I^{>}$, the first pairing changes by an element of $I$ by definition, while the second changes by an element of $J$ by Definition~\ref{def:induction}.  The same is true if $\eta$ is changed by an element of ${}_I E^{<}$, using Proposition~\ref{prop:induction-formulas}.  Hence the pairings are well-defined.  They remain contractive for the quotient norms: for example, given $[\eta]$ and $[\xi]$, choose representatives $\eta_0$ and $\xi_0$ with norms arbitrarily close to the corresponding quotient norms and use
\[
 \|\langle[\eta],[\xi]\rangle_{B/I}\|
 \leq \|\langle\eta_0,\xi_0\rangle_B\|
 \leq \|\eta_0\|\,\|\xi_0\|.
\]
The $A/J$-valued pairing is handled in the same way.

The Morita compatibility identities pass to the quotients.  Fullness follows because the images of the original full pairing spans are dense in $A/J$ and $B/I$.  Finally, nondegeneracy of the quotient module actions follows by taking images of the dense subspaces $AE^{>}$, $E^{>}B$, $BE^{<}$ and $E^{<}A$.  Thus the quotient pair is a Morita equivalence.
\end{proof}

\section{Dense algebraic cores and geometric ideals}

The preceding ideal correspondence is purely Banach-algebraic.  Inspired by \cite{Chung2025}, we now isolate a compatibility condition under which it preserves density relative to distinguished algebraic cores.

\begin{defn}\label{def:dense-subcontext}
Let $(E^{<},E^{>})$ be a Morita equivalence between Banach algebras $A$ and $B$.  A \emph{compatible dense subcontext} consists of dense subalgebras
\[
 A_0\subseteq A,
 \qquad
 B_0\subseteq B,
\]
and dense linear subspaces
\[
 E_0^{<}\subseteq E^{<},
 \qquad
 E_0^{>}\subseteq E^{>},
\]
satisfying the four module-stability conditions
\begin{align*}
 B_0E_0^{<}&\subseteq E_0^{<},&
 E_0^{<}A_0&\subseteq E_0^{<},\\
 A_0E_0^{>}&\subseteq E_0^{>},&
 E_0^{>}B_0&\subseteq E_0^{>},
\end{align*}
and the pairing conditions
\[
 \langle E_0^{<},E_0^{>}\rangle_B\subseteq B_0,
 \qquad
 {}_A\langle E_0^{>},E_0^{<}\rangle\subseteq A_0.
\]
\end{defn}

We impose the module stability conditions in Definition~\ref{def:dense-subcontext} separately because $A_0$ and $B_0$ need not be unital; a triple-product condition alone does not provide the required stability used in the proof below.

\begin{defn}\label{def:core-geometric}
Let $A_0\subseteq A$ be a dense subalgebra.  A closed ideal $J\triangleleft A$ is called \emph{$A_0$-geometric} if
\[
 \overline{J\cap A_0}=J.
\]
We write
\[
 \mathcal G(A;A_0)
 :=\{J\triangleleft A:J\text{ is }A_0\text{-geometric}\}.
\]
\end{defn}

\begin{thm}[Preservation of geometric ideals]\label{thm:geometric-preservation}
Let $A$ and $B$ have bounded approximate identities, let $(E^{<},E^{>})$ be a Morita equivalence between them, and suppose $(A_0,B_0,E_0^{<},E_0^{>})$ is a compatible dense subcontext.  Then, for every closed ideal $I\triangleleft B$,
\[
 I\in\mathcal G(B;B_0)
 \quad\Longleftrightarrow\quad
 \Ind_B^A(I)\in\mathcal G(A;A_0).
\]
Consequently, Morita induction restricts to an order isomorphism
\[
 \Ind_B^A:\mathcal G(B;B_0)
 \xrightarrow{\cong}
 \mathcal G(A;A_0).
\]
\end{thm}

\begin{proof}
Let $I\in\mathcal G(B;B_0)$ and set $J=\Ind_B^A(I)$.  By Proposition~\ref{prop:induction-formulas},
\[
 J
 =\overline{\spann\bigl\{{}_A\langle\xi b,\eta\rangle:
 \xi\in E^{>},\ b\in I,\ \eta\in E^{<}\bigr\}}.
\]
Set
\[
 D_I
 :=\spann\bigl\{{}_A\langle\xi b,\eta\rangle:
 \xi\in E_0^{>},\ b\in I\cap B_0,\ \eta\in E_0^{<}\bigr\}.
\]
Since $E_0^{>}B_0\subseteq E_0^{>}$ and the $A$-valued pairing maps $E_0^{>}\times E_0^{<}$ into $A_0$, we have $D_I\subseteq A_0$.  Since $b\in I$, every element of $D_I$ also belongs to $J$.  Hence
\[
 D_I\subseteq J\cap A_0.
\]

We claim that $D_I$ is dense in $J$.  Fix ${}_A\langle\xi b,\eta\rangle$ with $\xi\in E^{>},\ b\in I,\ \eta\in E^{<}$.  Choose sequences
\[
 \xi_n\in E_0^{>},\quad \xi_n\to\xi,
 \qquad
 b_n\in I\cap B_0,\quad b_n\to b,
 \qquad
 \eta_n\in E_0^{<},\quad \eta_n\to\eta.
\]
The map
\[
 E^{>}\times B\times E^{<}\longrightarrow A,
 \qquad
 (\xi,b,\eta)\longmapsto{}_A\langle\xi b,\eta\rangle
\]
is continuous trilinear, so
\[
 {}_A\langle\xi_n b_n,\eta_n\rangle
 \longrightarrow{}_A\langle\xi b,\eta\rangle.
\]
Thus $\overline{D_I}=J$.  Therefore
\[
 J=\overline{D_I}
 \subseteq\overline{J\cap A_0}
 \subseteq J,
\]
so $J$ is $A_0$-geometric.

The converse follows by applying the same argument to the inverse Morita equivalence and using Theorem~\ref{thm:ideal-lattice}.  The final assertion is then immediate.
\end{proof}

For a closed ideal $J\triangleleft A$, define its
\emph{$A_0$-geometric core} by
\[
 \operatorname{gcore}_{A_0}(J)
 :=
 \overline{J\cap A_0}.
\]
Thus $\operatorname{gcore}_{A_0}(J)$ is the part of $J$ detected by
the distinguished algebraic core $A_0$.

\begin{prop}[Preservation of geometric cores]
\label{prop:geometric-core-preservation}
Under the hypotheses of
Theorem~\ref{thm:geometric-preservation},
$\operatorname{gcore}_{A_0}(J)$ is the largest
$A_0$-geometric closed ideal contained in $J$, for every closed ideal
$J\triangleleft A$.  Moreover, for every closed ideal
$I\triangleleft B$,
\[
 \Ind_B^A\bigl(\operatorname{gcore}_{B_0}(I)\bigr)
 =
 \operatorname{gcore}_{A_0}\bigl(\Ind_B^A(I)\bigr).
\]
\end{prop}

\begin{proof}
Let
\[
 G:=\operatorname{gcore}_{A_0}(J)
   =\overline{J\cap A_0}.
\]
We first check that $G$ is a closed two-sided ideal of $A$.
Let $a\in A$ and $x\in G$.  Choose sequences
\[
 a_n\in A_0,\qquad a_n\to a,
\]
and
\[
 x_n\in J\cap A_0,\qquad x_n\to x.
\]
Since $J\cap A_0$ is a two-sided ideal of $A_0$, we have $a_nx_n,\ x_na_n\in J\cap A_0$.
Passing to the limit gives $ax,\ xa\in G$.
Thus $G\triangleleft A$.

Moreover, $J\cap A_0\subseteq G\cap A_0$, and hence
\[
 G=\overline{J\cap A_0}
 \subseteq\overline{G\cap A_0}
 \subseteq G.
\]
Therefore $G$ is $A_0$-geometric.

If $L\triangleleft A$ is $A_0$-geometric and $L\subseteq J$, then
\[
 L
 =
 \overline{L\cap A_0}
 \subseteq
 \overline{J\cap A_0}
 =
 G.
\]
Thus $G$ is the largest $A_0$-geometric ideal contained in $J$.

Now put
\[
 \Phi:=\Ind_B^A
\]
and
\[
 K:=\operatorname{gcore}_{B_0}(I).
\]
By Theorem~\ref{thm:geometric-preservation}, $\Phi(K)$ is
$A_0$-geometric, and since $K\subseteq I$, we have $\Phi(K)\subseteq\Phi(I)$. 
By maximality of the geometric core,
\[
 \Phi(K)
 \subseteq
 \operatorname{gcore}_{A_0}(\Phi(I)).
\]

Conversely, let
\[
 L:=\operatorname{gcore}_{A_0}(\Phi(I)).
\]
Then $\Phi^{-1}(L)$ is $B_0$-geometric by
Theorem~\ref{thm:geometric-preservation}, and $\Phi^{-1}(L)\subseteq I$.
Hence, by maximality of $K$,
\[
 \Phi^{-1}(L)\subseteq K.
\]
Applying $\Phi$ gives
\[
 L\subseteq\Phi(K).
\]
The two inclusions prove the asserted equality.
\end{proof}

\begin{rem}\label{rem:geometric-core-interior}
For a fixed dense subalgebra $A_0\subseteq A$, the map
\[
 \gamma_{A_0}:\mathcal I(A)\longrightarrow\mathcal I(A),
 \qquad
 \gamma_{A_0}(J):=\operatorname{gcore}_{A_0}(J),
\]
is order preserving, satisfies
\[
 \gamma_{A_0}(J)\subseteq J,
 \qquad
 \gamma_{A_0}^2=\gamma_{A_0},
\]
and has
\[
 \mathcal G(A;A_0)
 =
 \{J\in\mathcal I(A):\gamma_{A_0}(J)=J\}
\]
as its set of fixed points.  Thus $\gamma_{A_0}$ is an order-theoretic interior
operator on the ideal lattice.

If
\[
 \Phi:=\Ind_B^A,
\]
then Proposition~\ref{prop:geometric-core-preservation} may be
written as
\[
 \Phi\circ\gamma_{B_0}
 =
 \gamma_{A_0}\circ\Phi.
\]
Consequently, for every $B_0$-geometric ideal $G\triangleleft B$,
Morita induction restricts to an order isomorphism
\[
 \{I\triangleleft B:
   \operatorname{gcore}_{B_0}(I)=G\}
 \xrightarrow{\cong}
 \{J\triangleleft A:
   \operatorname{gcore}_{A_0}(J)=\Phi(G)\}.
\]
Thus Morita induction preserves the decomposition of the ideal
lattices into blocks determined by their geometric cores.
\end{rem}

\begin{prop}[Join closure of geometric ideals]
\label{prop:geometric-join-closure}
For every dense subalgebra $A_0\subseteq A$, the family
$\mathcal G(A;A_0)$ is closed under arbitrary joins taken in
$\mathcal I(A)$.
\end{prop}

\begin{proof}
Let
\[
 J
 :=
 \bigvee_{\lambda\in\Lambda}J_\lambda,
 \qquad
 J_\lambda\in\mathcal G(A;A_0).
\]
Each $J_\lambda$ is an $A_0$-geometric ideal contained in $J$.
By the maximality statement in
Proposition~\ref{prop:geometric-core-preservation},
\[
 J_\lambda
 \subseteq
 \operatorname{gcore}_{A_0}(J)
 \qquad(\lambda\in\Lambda).
\]
Taking the join gives
\[
 J
 \subseteq
 \operatorname{gcore}_{A_0}(J)
 \subseteq J.
\]
Hence
\[
 J=\operatorname{gcore}_{A_0}(J),
\]
so $J$ is $A_0$-geometric.
\end{proof}

\begin{rem}\label{rem:geometric-lattice}
By Proposition~\ref{prop:geometric-join-closure}, no additional
hypothesis is needed for joins.  Thus
$\mathcal G(A;A_0)$ is a sublattice of $\mathcal I(A)$ precisely
when it is closed under finite intersections.  Under the hypotheses
of Theorem~\ref{thm:geometric-preservation}, this meet-closure
property holds on the $A$-side if and only if it holds on the
$B$-side.  Indeed, Morita induction preserves both intersections in
the full ideal lattices and geometricity.  In this case the
restricted induction map
\[
 \Ind_B^A:
 \mathcal G(B;B_0)
 \xrightarrow{\cong}
 \mathcal G(A;A_0)
\]
is a lattice isomorphism.
\end{rem}

\section{Characterization via localized submodules}

The dense core theorem applies to ideals defined by approximation
from a fixed algebraic core.  A complementary description applies
when ideals are detected by the closed submodules they generate in a
Morita bimodule.  In fact, equality of these localized submodules
characterizes the Morita correspondence completely.

\begin{prop}[Localized submodule characterization]
\label{prop:localized-submodule}
Let $A$ and $B$ be Banach algebras with bounded approximate
identities, and let $(E^{<},E^{>})$ be a Morita equivalence between
them. Let $J\triangleleft A$ and $I\triangleleft B$ be closed ideals.
Then the following are equivalent:
\begin{enumerate}
\item
$J$ and $I$ correspond under Morita induction, i.e.
\[
 \Ind_A^B(J)=I
 \qquad\text{and}\qquad
 \Ind_B^A(I)=J;
\]
\item
$\overline{JE^{>}} = \overline{E^{>}I}$.
\end{enumerate}
\end{prop}

\begin{proof}
Suppose first that (i) holds.  Since
$J=\Ind_B^A(I)$, Lemma~\ref{lem:induced-module} and
Proposition~\ref{prop:ideal-submodules} give
\[
 \overline{JE^{>}}
 =E_I^{>}
 =\overline{E^{>}I},
\]
so (ii) holds.

Conversely, assume (ii).  By the symmetric form of
Proposition~\ref{prop:induction-formulas},
\[
 \Ind_A^B(J)
 =\clspan{\langle E^{<},JE^{>}\rangle_B}.
\]
Hence
\[
 \Ind_A^B(J)
 =\clspan{\langle E^{<},E^{>}I\rangle_B}
 =\clspan{\langle E^{<},E^{>}\rangle_B I}.
\]
Fullness gives
\[
 \clspan{\langle E^{<},E^{>}\rangle_B}=B,
\]
and therefore
\[
 \Ind_A^B(J)=\overline{BI}=I,
\]
where the last equality follows from the bounded approximate identity
of $B$.  Theorem~\ref{thm:ideal-lattice} now gives
$\Ind_B^A(I)=J$.
\end{proof}

\section{Dynamical ideals of reduced groupoid \texorpdfstring{$L^p$}{Lp}-operator algebras}
\label{sec:groupoid-dynamical-ideals}

We now apply the localized submodule characterization to dynamical ideals of reduced groupoid $L^p$-operator algebras.  We recall the groupoid vocabulary
needed for the argument, but refer to \cite{GardellaLupini2017,BardadynKwasniewskiMcKee2025}
for the construction and basic properties of reduced groupoid
$L^p$-operator algebras.

Throughout this section, $G$ and $H$ are locally compact Hausdorff
\'etale groupoids.  We use standard notation and terminology for
groupoids, their actions and equivalences; see
\cite{Paterson1999,Williams2019}.  We write $G^{(0)}$ for the unit
space of $G$, and $r,s$ for its range and source maps.

\begin{defn}\label{def:groupoid-equivalence-dynamical}
A \emph{$(G,H)$-equivalence} is a locally compact Hausdorff space
$Z$ carrying commuting free and proper left $G$- and right
$H$-actions, with anchor maps
\[
 r_Z:Z\longrightarrow G^{(0)},
 \qquad
 s_Z:Z\longrightarrow H^{(0)},
\]
such that the induced maps give homeomorphisms
\[
 Z/H\cong G^{(0)},
 \qquad
 G\backslash Z\cong H^{(0)}.
\]
\end{defn}

This is the usual notion of groupoid equivalence; see
\cite{MuhlyRenaultWilliams1987,Williams2019}.  In particular,
$r_Z$ and $s_Z$ are surjective,
\[
 r_Z(z\cdot\eta)=r_Z(z),
 \qquad
 s_Z(\gamma\cdot z)=s_Z(z)
\]
whenever the indicated actions are defined, and the fibres of
$s_Z$ and $r_Z$ are precisely the $G$- and $H$-orbits,
respectively.  Thus, if $s_Z(z)=s_Z(z')$, there is a unique
$\gamma\in G$ such that $z'=\gamma\cdot z$, and symmetrically on
the right.

An open subset $U\subseteq G^{(0)}$ is \emph{invariant} if
\[
 r(\gamma)\in U \quad\Longrightarrow\quad s(\gamma)\in U
 \qquad(\gamma\in G).
\]
Equivalently, $r(\gamma)\in U$ if and only if $s(\gamma)\in U$.  For such
$U$, the reduction
\[
 G|_U=\{\gamma\in G:r(\gamma),s(\gamma)\in U\}
\]
is an open subgroupoid of $G$.
Let $\mathcal O(G)$ and $\mathcal O(H)$ denote the respective lattices
of open invariant subsets of $G^{(0)}$ and $H^{(0)}$.

\begin{lem}\label{lem:invariant-open-correspondence}
Let $Z$ be a $(G,H)$-equivalence.  For an open invariant subset
$U\subseteq G^{(0)}$, set
\[
 \Theta(U):=s_Z\bigl(r_Z^{-1}(U)\bigr)\subseteq H^{(0)}.
\]
Then $\Theta(U)$ is open and invariant, and
\begin{equation}\label{eq:anchor-localization}
 r_Z^{-1}(U)=s_Z^{-1}(\Theta(U)).
\end{equation}
The map
\[
 \Theta:\mathcal O(G)\longrightarrow\mathcal O(H)
\]
is a lattice isomorphism between the lattices of open invariant subsets, with
inverse
\[
 V\longmapsto r_Z\bigl(s_Z^{-1}(V)\bigr).
\]
\end{lem}

\begin{proof}
Since $r_Z^{-1}(U)$ is open and $s_Z$ is open, $\Theta(U)$ is open.  To see
that it is invariant, suppose that $\eta\in H$ and
$r(\eta)\in\Theta(U)$.  Choose $z\in Z$ with
$r_Z(z)\in U$ and $s_Z(z)=r(\eta)$.  Then $z\cdot\eta$ is defined and
\[
 r_Z(z\cdot\eta)=r_Z(z)\in U,
 \qquad
 s_Z(z\cdot\eta)=s(\eta),
\]
so $s(\eta)\in\Theta(U)$.

The inclusion
$r_Z^{-1}(U)\subseteq s_Z^{-1}(\Theta(U))$ follows from the definition.
Conversely, let $z'\in s_Z^{-1}(\Theta(U))$.  There is $z\in Z$ such that
$r_Z(z)\in U$ and $s_Z(z)=s_Z(z')$.  Since the fibres of $s_Z$ are the
$G$-orbits, $z'=\gamma\cdot z$ for some $\gamma\in G$.  Hence
\[
 s(\gamma)=r_Z(z)\in U,
 \qquad
 r(\gamma)=r_Z(z').
\]
Invariance of $U$ gives $r_Z(z')\in U$.  This
proves \eqref{eq:anchor-localization}.

Define
\[
 \Psi(V):=r_Z\bigl(s_Z^{-1}(V)\bigr)
\]
for open invariant $V\subseteq H^{(0)}$.  The symmetric argument shows that
$\Psi(V)$ is open and invariant and that
$s_Z^{-1}(V)=r_Z^{-1}(\Psi(V))$.  Since both anchor maps are surjective,
\eqref{eq:anchor-localization} gives
\[
 \Psi(\Theta(U))=U,
 \qquad
 \Theta(\Psi(V))=V.
\]
Both maps preserve inclusions, and hence they are mutually inverse lattice
isomorphisms.
\end{proof}

\subsection{Dynamical ideals and inner support}
\label{subsec:dynamical-inner-support}

For $p\in[1,\infty]$, let $F^p_{\mathrm{red}}(G)$ denote the reduced
$L^p$-operator algebra of $G$.  If $U\subseteq G^{(0)}$ is open and
invariant, then $C_c(G|_U)$ is a two-sided ideal of $C_c(G)$, and therefore
\[
 I_U^G:=\overline{C_c(G|_U)}^{\,F^p_{\mathrm{red}}(G)}
\]
is a closed two-sided ideal of $F^p_{\mathrm{red}}(G)$.

The following basic facts are also proved in
\cite[Lemma~5.6]{ChungDu2026}; we include the short argument
needed here in order to keep the preservation results below
self-contained.

\begin{lem}\label{lem:unit-space-intersection}
Let $U\subseteq G^{(0)}$ be open and invariant.  Then
\[
 I_U^G\cap C_0(G^{(0)})=C_0(U).
\]
Consequently,
\[
 \insupp_G(I_U^G)=U,
\]
and the map
\[
 U\longmapsto I_U^G
\]
from open invariant subsets of $G^{(0)}$ to closed ideals of
$F^p_{\mathrm{red}}(G)$ is injective.
\end{lem}

\begin{proof}
Let $u\in G^{(0)}\setminus U$, and let $\lambda_u^p$ denote the
regular representation of $C_c(G)$ on $\ell^p(G_u)$.  Since $U$ is invariant and $u\notin U$, every
$\alpha\in G_u=s^{-1}(u)$ satisfies $r(\alpha)\notin U$.
Thus, if $\alpha,\gamma\in G_u$, then
\[
 r(\alpha)\notin U,
 \qquad
 r(\gamma)\notin U,
\]
and therefore the arrow $\alpha\gamma^{-1}$ does not belong to
$G|_U$.  It follows directly from the formula for the regular
representation that
\[
 \lambda_u^p(f)=0
 \qquad
 (f\in C_c(G|_U)).
\]
By continuity,
\[
 \lambda_u^p(a)=0
 \qquad
 (a\in I_U^G).
\]
If $a\in I_U^G\cap C_0(G^{(0)})$, then the regular representation acts by multiplication on the unit
space and therefore
\[
 \lambda_u^p(a)\delta_u=a(u)\delta_u.
\]
Thus $a(u)=0$.  Since $u\notin U$ was arbitrary,
\[
 I_U^G\cap C_0(G^{(0)})\subseteq C_0(U).
\]

Conversely,
\[
 C_c(U)\subseteq C_c(G|_U).
\]
For $f\in C_c(G^{(0)})$, the regular representations act by
multiplication, and hence
\[
 \|f\|_{F^p_{\mathrm{red}}(G)}=\|f\|_\infty.
\]
It follows by taking closures that
\[
 C_0(U)\subseteq I_U^G\cap C_0(G^{(0)}).
\]
This proves the equality.

The equality immediately gives
\[
 \insupp_G(I_U^G)=U,
\]
and hence $I_U^G=I_V^G$ implies $U=V$.
\end{proof}

\begin{defn}\label{def:inner-support}
Let $K\triangleleft F^p_{\mathrm{red}}(G)$ be a closed ideal.  Its
\emph{inner support} is
\[
 \insupp_G(K)
 :=
 \bigcup_{f\in K\cap C_0(G^{(0)})}
 \{u\in G^{(0)}:f(u)\neq0\}.
\]
We call $K$ \emph{dynamical} if it is the closed two-sided ideal
generated by $K\cap C_0(G^{(0)})$.
\end{defn}

The next lemma is the intrinsic description of dynamical ideals used
below; compare \cite[Lemma~5.7 and Proposition~5.9]{ChungDu2026}.

\begin{lem}\label{lem:dynamical-inner-support}
Let $K\triangleleft F^p_{\mathrm{red}}(G)$ be a closed ideal and put
\[
 U_K:=\insupp_G(K).
\]
Then $U_K$ is open and invariant,
\[
 K\cap C_0(G^{(0)})=C_0(U_K),
\]
and $K$ is dynamical if and only if
\[
 K=I_{U_K}^G.
\]
Moreover, $I_{U_K}^G$ is the largest dynamical ideal contained in
$K$.
\end{lem}

\begin{proof}
The set $U_K$ is open by definition.  Its invariance follows from
the usual bisection argument: if $x\in U_K$, $s(\gamma)=x$, and
$f\in K\cap C_0(G^{(0)})$ satisfies $f(x)\neq0$, choose an open
bisection $B$ containing $\gamma$ and $h\in C_c(B)$ with
$h(\gamma)=1$.  Let
\[
 h^*(\alpha):=\overline{h(\alpha^{-1})}
\]
denote the usual algebraic involution on $C_c(G)$; we use it only
inside the convolution algebra $C_c(G)$, and no involution on
$F^p_{\mathrm{red}}(G)$ is being invoked. Then
\[
 h*f*h^*\in K\cap C_0(G^{(0)})
\]
and
\[
 (h*f*h^*)(r(\gamma))=f(x)\neq0.
\]
Thus $r(\gamma)\in U_K$.

Since $K\cap C_0(G^{(0)})$ is a closed ideal of $C_0(G^{(0)})$,
the standard ideal correspondence for $C_0(G^{(0)})$ gives
\[
 K\cap C_0(G^{(0)})=C_0(U_K).
\]

By Lemma~\ref{lem:unit-space-intersection}, for every open invariant
$U\subseteq G^{(0)}$ one has
\[
 I_U^G\cap C_0(G^{(0)})=C_0(U)
 \qquad\text{and}\qquad
 \insupp_G(I_U^G)=U.
\]
For an open invariant subset $U\subseteq G^{(0)}$, the ideal
$I_U^G$ is precisely the closed ideal generated by $C_0(U)$.
Indeed, if $f\in C_c(G|_U)$, choose $h\in C_c(U)$ such that
$h=1$ on $r(\supp(f))$.  Then $h*f=f$.
It follows that $K$ is dynamical if and only if $K=I_{U_K}^G$.

Finally, if $I_V^G\subseteq K$ is dynamical, then
\[
 C_0(V)
 =I_V^G\cap C_0(G^{(0)})
 \subseteq K\cap C_0(G^{(0)})
 =C_0(U_K),
\]
so $V\subseteq U_K$.  Hence $I_V^G\subseteq I_{U_K}^G$.
\end{proof}

\begin{cor}[Parametrization of dynamical ideals]
\label{cor:dynamical-parametrization}
The assignment
\[
 U\longmapsto I_U^G
\]
is an order isomorphism from $\mathcal O(G)$ onto the family of
dynamical ideals of $F^p_{\mathrm{red}}(G)$.  Moreover, the dynamical ideals form a sublattice of the full ideal
lattice and are closed under arbitrary joins.  For open invariant
subsets $U,V\subseteq G^{(0)}$,
\[
 I_U^G\cap I_V^G
 =
 I_{U\cap V}^G,
 \qquad
 I_U^G\vee I_V^G
 =
 I_{U\cup V}^G,
\]
and, more generally, for every family
$(U_\lambda)_{\lambda\in\Lambda}$ of open invariant subsets,
\[
 \bigvee_{\lambda\in\Lambda} I_{U_\lambda}^G
 =
 I_{\bigcup_{\lambda\in\Lambda}U_\lambda}^G,
\]
where the meet and join are taken in
$\mathcal I(F^p_{\mathrm{red}}(G))$.  Thus
\[
 \mathcal O(G)
 \xrightarrow{\cong}
 \{I\triangleleft F^p_{\mathrm{red}}(G):
   I\text{ is dynamical}\}
\]
is a lattice isomorphism for the lattice operations inherited from
the full ideal lattice.
\end{cor}

\begin{proof}
The assignment $U\mapsto I_U^G$ is order preserving and injective by
Lemma~\ref{lem:unit-space-intersection}.  Conversely,
Lemma~\ref{lem:dynamical-inner-support} shows that every dynamical
ideal $K$ has the form
\[
 K=I_{\insupp_G(K)}^G.
\]
Thus the inverse map is
\[
 K\longmapsto\insupp_G(K),
\]
which is also order preserving.

It remains only to identify the ambient lattice operations.  Since
$I_U^G$ and $I_V^G$ are generated by $C_0(U)$ and $C_0(V)$,
respectively, and
\[
 \overline{C_0(U)+C_0(V)}
 =
 C_0(U\cup V),
\]
we have
\[
 I_U^G\vee I_V^G=I_{U\cup V}^G.
\]
The same argument applies to an arbitrary family.  Indeed, if
\[
 U:=\bigcup_{\lambda\in\Lambda}U_\lambda,
\]
then $U$ is open and invariant, and
\[
 \overline{
 \spann\bigl\{
 C_0(U_\lambda):
 \lambda\in\Lambda
 \bigr\}}
 =
 C_0(U).
\]
Hence the closed ideal generated by the family
$\{I_{U_\lambda}^G:\lambda\in\Lambda\}$ is precisely $I_U^G$, and
therefore
\[
 \bigvee_{\lambda\in\Lambda} I_{U_\lambda}^G
 =
 I_{\bigcup_{\lambda\in\Lambda}U_\lambda}^G.
\]

Clearly
\[
 I_{U\cap V}^G\subseteq I_U^G\cap I_V^G.
\]
For the reverse inclusion, let
$a\in I_U^G\cap I_V^G$.  The usual unit-space cutoff construction
gives a contractive approximate identity $(h_\lambda)$ for
$I_V^G$ with $h_\lambda\in C_c(V)$.  Hence $h_\lambda*a\longrightarrow a$.
For fixed $\lambda$, approximate $a$ by
$f_n\in C_c(G|_U)$.  Since $V$ is invariant, $h_\lambda*f_n\in C_c(G|_{U\cap V})$,
and therefore $h_\lambda*a\in I_{U\cap V}^G$.
Letting $\lambda$ tend to infinity proves $a\in I_{U\cap V}^G$.
\end{proof}

For a closed ideal
$K\triangleleft F^p_{\mathrm{red}}(G)$, define its
\emph{dynamical core} by
\[
 \dcore_G(K)
 :=
 I_{\insupp_G(K)}^G.
\]
We use the same notation with $G$ replaced by any groupoid under
consideration; in particular, $\dcore_H$ denotes the corresponding
map for $H$.
By Lemma~\ref{lem:dynamical-inner-support},
$\dcore_G(K)$ is the largest dynamical ideal contained in $K$.
Consequently, the map
\[
 \dcore_G:
 \mathcal I(F^p_{\mathrm{red}}(G))
 \longrightarrow
 \mathcal I(F^p_{\mathrm{red}}(G))
\]
is order preserving and satisfies
\[
 \dcore_G(K)\subseteq K,
 \qquad
 \dcore_G^2=\dcore_G.
\]
Its fixed points are precisely the dynamical ideals.
Thus $\dcore_G$ is the dynamical analogue of the
order-theoretic interior operator
$\operatorname{gcore}_{A_0}$ from Section~3.

We will also use the fact that $F^p_{\mathrm{red}}(G)$ has a
contractive approximate identity contained in $C_c(G^{(0)})$.  Indeed,
direct the compact subsets $K\subseteq G^{(0)}$ by inclusion and choose
$h_K\in C_c(G^{(0)})$ such that $0\leq h_K\leq1$ and $h_K=1$ on $K$.
A function on the unit space acts in every regular representation by
multiplication, so
\[
 \|h_K\|_{F^p_{\mathrm{red}}(G)}\leq1.
\]
For $f\in C_c(G)$, once $K$ contains
$r(\supp(f))\cup s(\supp(f))$, one has
$h_K*f=f=f*h_K$.  Density of $C_c(G)$ gives the claim.  The same observation
applies to $H$.

\subsection{Preservation under groupoid equivalence}
\label{subsec:groupoid-ideal-preservation}

We next recall only the part of the standard groupoid-equivalence Morita
context which is needed below.  Put
\[
 A=F^p_{\mathrm{red}}(G),
 \qquad
 B=F^p_{\mathrm{red}}(H).
\]
Suppose that the standard algebraic bimodule $C_c(Z)$ completes to the
$A$--$B$ bimodule $E^{>}$ in a Morita equivalence between $A$ and $B$, and
that the completed actions extend the algebraic formulas
\begin{align}
 (a\cdot\xi)(z)
 &=\sum_{r(\gamma)=r_Z(z)}
      a(\gamma)\xi(\gamma^{-1}\cdot z), \notag
 \\
 (\xi\cdot b)(z)
 &=\sum_{r(\eta)=s_Z(z)}
      \xi(z\cdot\eta)b(\eta^{-1}), \notag
\end{align}
for $a\in C_c(G)$, $b\in C_c(H)$ and $\xi\in C_c(Z)$.  These are the usual
module actions associated with a groupoid equivalence.  

\begin{prop}\label{prop:groupoid-dynamical-localization}
Let $Z$ be a $(G,H)$-equivalence, let $p\in[1,\infty]$, and assume that the
standard bimodule $C_c(Z)$ completes to a Morita equivalence between
\[
 A=F^p_{\mathrm{red}}(G)
 \quad\text{and}\quad
 B=F^p_{\mathrm{red}}(H)
\]
as above.  If $U\subseteq G^{(0)}$ is open and invariant and
$V=\Theta(U)$ is the corresponding open invariant subset of $H^{(0)}$, then
\begin{equation} \notag
 \overline{I_U^G E^{>}}
 =
 \overline{E^{>}I_V^H}.
\end{equation}
Consequently,
\[
 \Ind_A^B(I_U^G)=I_V^H.
\]
Thus Morita induction carries dynamical ideals of $A$ onto dynamical ideals
of $B$.
\end{prop}

\begin{proof}
Let
\[
 Z_U:=r_Z^{-1}(U)=s_Z^{-1}(V),
\]
where the equality follows from Lemma~\ref{lem:invariant-open-correspondence}.
We first prove the algebraic localization identity
\begin{equation}\label{eq:algebraic-groupoid-localization}
 C_c(G|_U)\cdot C_c(Z)
 =C_c(Z_U)
 =C_c(Z)\cdot C_c(H|_V).
\end{equation}
Here $S\cdot T$ denotes the set of single module products of the form $st$ with $s\in S,t\in T$, rather than their linear span.

Let $a\in C_c(G|_U)$ and $\xi\in C_c(Z)$.  If
$(a\cdot\xi)(z)\neq0$, there is $\gamma\in G|_U$ with $r(\gamma)=r_Z(z)$.  Hence
$r_Z(z)\in U$, so $a\cdot\xi\in C_c(Z_U)$.  This gives
\[
 C_c(G|_U)\cdot C_c(Z)\subseteq C_c(Z_U).
\]
Conversely, let $\xi\in C_c(Z_U)$.  The set
$r_Z(\supp(\xi))$ is a compact subset of $U$.  Choose $h\in C_c(U)$ with
$h=1$ on $r_Z(\supp(\xi))$, and regard $h$ as an element of
$C_c(G|_U)$ supported on the unit space.  Then
\[
 (h\cdot\xi)(z)=h(r_Z(z))\xi(z)=\xi(z),
\]
so $\xi\in C_c(G|_U)\cdot C_c(Z)$.  This proves the first equality in
\eqref{eq:algebraic-groupoid-localization}.  The second is symmetric: for
$\xi\in C_c(Z_U)$ choose $k\in C_c(V)$ equal to $1$ on
$s_Z(\supp(\xi))$ and use
\[
 (\xi\cdot k)(z)=\xi(z)k(s_Z(z)).
\]

Let
\[
 E_U^{>}:=\overline{C_c(Z_U)}^{\,E^{>}}.
\]
The first equality in \eqref{eq:algebraic-groupoid-localization} gives
$E_U^{>}\subseteq\overline{I_U^G E^{>}}$.  For the reverse containment,
let $a\in I_U^G$ and $\xi\in E^{>}$.  Approximate $a$ by elements of
$C_c(G|_U)$ and $\xi$ by elements of $C_c(Z)$.  Continuity of the module
action shows that $a\xi$ is a norm limit of elements of $C_c(Z_U)$.
Therefore
\[
 \overline{I_U^G E^{>}}=E_U^{>}.
\]
The same argument on the right gives
\[
 \overline{E^{>}I_V^H}=E_U^{>}.
\]
This proves the asserted equality.  The localized
submodule characterization from the preceding section now gives
$\Ind_A^B(I_U^G)=I_V^H$.

Finally, Lemma~\ref{lem:invariant-open-correspondence} shows that every open
invariant $V\subseteq H^{(0)}$ equals $\Theta(U)$ for some open invariant
$U\subseteq G^{(0)}$.  Hence every dynamical ideal on the $H$-side occurs as
the Morita induction of one on the $G$-side.
\end{proof}

\begin{cor}\label{cor:dynamical-ideals-morita}
Under the hypotheses of
Proposition~\ref{prop:groupoid-dynamical-localization}, let
$U\subseteq G^{(0)}$ be open and invariant and put
$V=\Theta(U)$.  Then the corresponding dynamical ideals
\[
 I_U^G
 \qquad\text{and}\qquad
 I_V^H
\]
are Morita equivalent through the restricted bimodules.
In particular, the $A$--$B$ restricted module is
\[
 E_U^{>}
 =
 \overline{C_c(Z_U)}^{\,E^{>}},
 \qquad
 Z_U=r_Z^{-1}(U)=s_Z^{-1}(V).
\]
\end{cor}

\begin{proof}
The proof of
Proposition~\ref{prop:groupoid-dynamical-localization} gives
\[
 \overline{I_U^GE^{>}}
 =
 \overline{E^{>}I_V^H}
 =
 E_U^{>}.
\]
Moreover, $I_U^G$ has a contractive approximate identity contained
in $C_c(U)$, by the same argument used above for
$F^p_{\mathrm{red}}(G)$.  The conclusion therefore follows from
Theorem~\ref{thm:corresponding-ideals-morita}.
\end{proof}

\begin{cor}[Quotients by dynamical ideals]
\label{cor:groupoid-dynamical-quotients}
Under the hypotheses of
Proposition~\ref{prop:groupoid-dynamical-localization}, let
$U\subseteq G^{(0)}$ be open and invariant and put
\[
 V:=\Theta(U).
\]
Then
\[
 F^p_{\mathrm{red}}(G)/I_U^G
 \qquad\text{and}\qquad
 F^p_{\mathrm{red}}(H)/I_V^H
\]
are Morita equivalent.
\end{cor}

\begin{proof}
By Proposition~\ref{prop:groupoid-dynamical-localization},
$I_U^G$ and $I_V^H$ correspond under Morita induction.  The result
therefore follows from Theorem~\ref{thm:quotient-morita}.
\end{proof}

\begin{thm}[Preservation of inner support and dynamical cores]
\label{thm:inner-support-preservation}
Under the hypotheses of
Proposition~\ref{prop:groupoid-dynamical-localization}, let
\[
 \Phi:=\Ind_A^B:
 \mathcal I(A)\longrightarrow\mathcal I(B).
\]
Then, for every closed ideal $K\triangleleft A$,
\[
 \insupp_H(\Phi(K))
 =
 \Theta\bigl(\insupp_G(K)\bigr).
\]
Equivalently,
\[
 \Phi\bigl(\dcore_G(K)\bigr)
 =
 \dcore_H\bigl(\Phi(K)\bigr).
\]
\end{thm}

\begin{proof}
Put
\[
 U:=\insupp_G(K).
\]
By Lemma~\ref{lem:dynamical-inner-support}, $I_U^G$ is the largest
dynamical ideal contained in $K$.

Since $\Phi$ is an order isomorphism and
Proposition~\ref{prop:groupoid-dynamical-localization} gives
\[
 \Phi(I_U^G)=I_{\Theta(U)}^H,
\]
the ideal $I_{\Theta(U)}^H$ is the largest dynamical ideal contained
in $\Phi(K)$.  Applying
Lemma~\ref{lem:dynamical-inner-support} on the $H$-side gives
\[
 I_{\Theta(U)}^H
 =
 I_{\insupp_H(\Phi(K))}^H.
\]
The parametrization of dynamical ideals by invariant open subsets is
injective, so
\[
 \insupp_H(\Phi(K))=\Theta(U).
\]
Using
Proposition~\ref{prop:groupoid-dynamical-localization}, we also have
\begin{align*}
 \Phi\bigl(\dcore_G(K)\bigr)
 &=
 \Phi\bigl(I_{\insupp_G(K)}^G\bigr)\\
 &=
 I_{\Theta(\insupp_G(K))}^H\\
 &=
 I_{\insupp_H(\Phi(K))}^H\\
 &=
 \dcore_H\bigl(\Phi(K)\bigr),
\end{align*}
which proves the equivalent formulation.
\end{proof}

Thus Theorem~\ref{thm:inner-support-preservation} is the dynamical
counterpart of
Proposition~\ref{prop:geometric-core-preservation}: Morita induction
intertwines the corresponding interior operators that select the
largest distinguished ideal below an arbitrary closed ideal.

\begin{cor}[Inner-support blocks] \label{cor:inner-support-blocks}
For an open invariant subset $U\subseteq G^{(0)}$, put
\[
 \mathcal I_G[U]
 :=
 \{K\triangleleft F^p_{\mathrm{red}}(G):
   \insupp_G(K)=U\}.
\]
For an open invariant subset $V\subseteq H^{(0)}$, define
analogously
\[
 \mathcal I_H[V]
 :=
 \{L\triangleleft F^p_{\mathrm{red}}(H):
   \insupp_H(L)=V\}.
\]
Then Morita induction restricts to an order isomorphism
\[
 \mathcal I_G[U]
 \xrightarrow{\cong}
 \mathcal I_H[\Theta(U)].
\]
Moreover, these blocks have least elements
\[
 I_U^G
 \qquad\text{and}\qquad
 I_{\Theta(U)}^H,
\]
respectively, and Morita induction carries the former onto the
latter.
\end{cor}

\begin{proof}
Let
\[
 \Phi:=\Ind_A^B.
\]
By Theorem~\ref{thm:inner-support-preservation}, if
$K\in\mathcal I_G[U]$, then
\[
 \insupp_H(\Phi(K))
 =
 \Theta(\insupp_G(K))
 =
 \Theta(U),
\]
so
\[
 \Phi\bigl(\mathcal I_G[U]\bigr)
 \subseteq
 \mathcal I_H[\Theta(U)].
\]
Applying the same argument to the inverse Morita equivalence gives
the reverse inclusion.  Since $\Phi$ is an order isomorphism on the full ideal lattices by
Theorem~\ref{thm:ideal-lattice}, its restriction is an order
isomorphism between the two indicated inner-support blocks.

Finally, if $K\in\mathcal I_G[U]$, then
Lemma~\ref{lem:dynamical-inner-support} gives
\[
 I_U^G
 =
 \dcore_G(K)
 \subseteq K.
\]
Thus $I_U^G$ is the least element of $\mathcal I_G[U]$.
The same argument on the $H$-side shows that
$I_{\Theta(U)}^H$ is the least element of
$\mathcal I_H[\Theta(U)]$, and
Proposition~\ref{prop:groupoid-dynamical-localization} gives
\[
 \Phi(I_U^G)=I_{\Theta(U)}^H.
\]
\end{proof}

\begin{cor}\label{cor:dynamical-ideal-lattice}
Let $G$ and $H$ be locally compact Hausdorff \'{e}tale groupoids with
paracompact unit spaces, and let $Z$ be a $(G,H)$-equivalence.  Then,
for every $p\in[1,\infty]$, the Morita equivalence between
$F^p_{\mathrm{red}}(G)$ and $F^p_{\mathrm{red}}(H)$ constructed in
\cite{ChungDelfinWang2026} induces a lattice isomorphism
between their dynamical ideals. More precisely,
\[
 \Ind_{F^p_{\mathrm{red}}(G)}^{F^p_{\mathrm{red}}(H)}(I_U^G)
 =I_{\Theta(U)}^H
\]
for every open invariant $U\subseteq G^{(0)}$.
\end{cor}

\begin{proof}
In \cite{ChungDelfinWang2026}, the bimodule
$C_c(Z)$ is completed to a Morita equivalence between
$F^p_{\mathrm{red}}(G)$ and $F^p_{\mathrm{red}}(H)$; see \cite{SimsWilliams} for the $p=2$ case.
Proposition~\ref{prop:groupoid-dynamical-localization} therefore gives
\[
\Ind_{F^p_{\mathrm{red}}(G)}^{F^p_{\mathrm{red}}(H)}(I_U^G)
=
I_{\Theta(U)}^H .
\]
Since
\[
 \Theta:\mathcal O(G)\longrightarrow\mathcal O(H)
\]
is a lattice isomorphism by
Lemma~\ref{lem:invariant-open-correspondence}, and the
parametrizations of dynamical ideals on the two sides are lattice
isomorphisms by
Corollary~\ref{cor:dynamical-parametrization}, the result follows.
\end{proof}

Notice that the preservation argument uses only the density of
$C_c(Z)$, the standard algebraic module actions, and the
localization identity
\eqref{eq:algebraic-groupoid-localization}, together with
Proposition~\ref{prop:localized-submodule}.

\section{Projection approximate identities in the \texorpdfstring{$\ell^p$}{l-p} uniform algebra}

Let $X$ be a metric space with bounded geometry and $1\leq p<\infty$.  Thus, for every $r\geq0$, letting $B(x,r)$ denote the ball of radius $r$ about $x\in X$, we have
\[
 \sup_{x\in X}|B(x,r)|<\infty.
\]
Recall that
\[
 \ell^p(X,\ell^p)
 =\Bigl(\bigoplus_{x\in X}\ell^p\Bigr)_{\ell^p}.
\]
For $T\in B(\ell^p(X,\ell^p))$, write $T=(T_{xy})_{x,y\in X}$ with $T_{xy}\in B(\ell^p)$.  The propagation of $T$ is
\[
 \Prop(T)=\sup\{d(x,y):T_{xy}\neq0\}
\]
with the convention $\Prop(0)=0$.
Here and below, a \emph{projection} on a Banach space means a bounded idempotent operator; no self-adjointness is intended unless $p=2$ and we explicitly say ``orthogonal projection.''

Following \cite{Chung2025}, let $U\C^p[X]$ be the algebra of finite-propagation operators $T\in B(\ell^p(X,\ell^p))$ for which
\[
 \sup_{x,y\in X}\rank(T_{xy})<\infty,
\]
and define
\[
 UB^p(X)=\overline{U\C^p[X]}^{\|\cdot\|}.
\]

\subsection{The Hilbert space model and its Banach space analog}

The projection construction below is motivated by the Hilbert space case.  Suppose first that $p=2$ and $R_1,\ldots,R_n$ are finite-rank operators on a Hilbert space.  If $P$ is the orthogonal projection onto
\[
 \sum_i\ran R_i+\sum_i\ran R_i^*,
\]
then
\[
 PR_i=R_i=R_iP
\]
for every $i$, and $\|P\|=1$.  

For $p\neq2$, these two requirements are naturally encoded on two
different spaces.  The relation $PR=R$ is ensured by requiring $P$ to
fix $\ran R\subseteq\ell^p$, whereas $RP=R$ is ensured by requiring
$P^*$ to fix $\ran R^*\subseteq(\ell^p)^*$.  Thus the Banach space
problem is to construct a uniformly bounded finite-rank projection
whose action on $\ell^p$ and whose adjoint action on $(\ell^p)^*$ fix
prescribed finite-dimensional subspaces simultaneously.  The
quantitative projection argument recalled below provides exactly this.

For Roe $C^*$-algebras, approximate units of projections were obtained by Higson and Roe \cite[Lemma~3.8]{HigsonRoe1994}.  For the standard discrete Roe module $\ell^2(X,\ell^2)=\bigoplus_{x\in X}\ell^2$,
Ewert and Meyer give a particularly transparent construction using
diagonal projections whose restriction to the copy of $\ell^2$
indexed by $x$ is a finite-rank coordinate projection
\cite[Proposition~2.19]{EwertMeyer2019}.  The ranks of these
coordinate projections may depend on $x$. In $U\C^p[X]$, by contrast, the ranks of the matrix
entries must be uniformly bounded.  This is the reason for the quantitative Banach space argument below.

\subsection{The uniform projection property}

We use the \emph{uniform projection property} introduced by
Pe\l czy\'nski and Rosenthal \cite{PelczynskiRosenthal1975}.
It is also called the \emph{uniform projection approximation
property}, in particular in \cite{Johnson1980} and \cite[Section~7]{Casazza2001}.  

\begin{defn}\label{def:upp}
A Banach space $V$ has the \emph{uniform projection property} if there exist a constant $K\geq1$ and a function $f:\N\to\N$ such that, for every finite-dimensional subspace $F\subseteq V$, there is a finite-rank projection $P\in B(V)$ satisfying
\[
 F\subseteq\ran P,
 \qquad
 \|P\|\leq K,
 \qquad
 \rank P\leq f(\dim F).
\]
\end{defn}

Pe\l czy\'nski and Rosenthal proved that every $L^r(\mu)$,
$1\leq r\leq\infty$, has the $1+$-uniform projection property;
see \cite{PelczynskiRosenthal1975} and the discussion in
\cite[Section~9]{Heinrich1980}.  Here $1+$ means that, for every
$\varepsilon>0$, the projections in
Definition~\ref{def:upp} may be chosen with norm at most
$1+\varepsilon$, with the rank bound allowed to depend on
$\varepsilon$.  In particular, $\ell^p$ has the uniform projection
property for every $1\leq p<\infty$.

\begin{lem}
\label{lem:quantitative-primal-dual-projection}
For each $1\leq p<\infty$, there exist a constant $C_p\geq1$ and a
function $h_p:\N\to\N$ such that the following holds.  Let $n\geq1$, and
let
\[
 E\subseteq\ell^p,
 \qquad
 G\subseteq(\ell^p)^*
\]
be finite-dimensional subspaces with
\[
 \dim E,\dim G\leq n.
\]
Then there is a finite-rank projection $P\in B(\ell^p)$ such that
\[
 P|_E=I_E,
 \qquad
 P^*|_G=I_G,
 \qquad
 \|P\|\leq C_p,
 \qquad
 \rank P\leq h_p(n).
\]
\end{lem}

\begin{proof}
We record the quantitative content of the proof of
\cite[Proposition~1.1]{Johnson1980}.  Choose
$K_p^{(0)}\geq1$ and
$f_p^{(0)}:\N\to\N$ for the uniform projection property of
$\ell^p$.

The uniform projection property implies the uniform approximation
property.  By the self-duality of the uniform approximation property
due to Heinrich \cite{Heinrich1980}, as used in the proof of
\cite[Proposition~1.1]{Johnson1980}, there are quantitative data
$K_p^{(1)}\geq1$ and $f_p^{(1)}:\N\to\N$ such that, whenever
\[
 G\subseteq(\ell^p)^*,
 \qquad
 \dim G\leq n,
\]
there is a finite-rank operator
\[
 T\in B((\ell^p)^*)
\]
satisfying
\[
 T|_G=I_G,
 \qquad
 \|T\|\leq K_p^{(1)},
 \qquad
 \rank T\leq f_p^{(1)}(n).
\]
The local-reflexivity step in Johnson's proof allows $T$ to be
chosen in the form
\[
 T=S^*
\]
for a finite-rank operator $S\in B(\ell^p)$.  After increasing the
quantitative data if necessary, we therefore have
\[
 S^*|_G=I_G,
 \qquad
 \|S\|\leq K_p^{(1)},
 \qquad
 \rank S\leq f_p^{(1)}(n).
\]
Replacing $f_p^{(0)}$ and $f_p^{(1)}$, if necessary, by their
nondecreasing envelopes
\[
 n\longmapsto
 \max_{1\leq k\leq n} f_p^{(i)}(k),
 \qquad i=0,1,
\]
we may and shall assume that both functions are nondecreasing.
Set
\[
 K_p:=\max\{K_p^{(0)},K_p^{(1)}\},
 \qquad
 f_p(n):=\max\{f_p^{(0)}(n),f_p^{(1)}(n)\}.
\]

By enlarging $E$ and $G$ if necessary, we may assume that
\[
 \dim E=\dim G=n.
\]
Choose $S$ as above.  Then
\[
 S^*|_G=I_G,
 \qquad
 \|S\|\leq K_p,
 \qquad
 \rank S\leq f_p(n).
\]

Since
\[
 \dim\bigl(E+S(\ell^p)\bigr)
 \leq n+f_p(n),
\]
the uniform projection property of $\ell^p$ gives a finite-rank
projection $Q\in B(\ell^p)$ such that
\[
 Q|_{E+S(\ell^p)}=I,
 \qquad
 \|Q\|\leq K_p,
 \qquad
 \rank Q\leq f_p\bigl(n+f_p(n)\bigr).
\]
In particular,
\[
 QS=S.
\]
Set
\[
 P:=S+Q-SQ.
\]
Since $Q^2=Q$ and $QS=S$, one has
\[
 QP=P,
 \qquad
 PQ=Q.
\]
It follows that $P$ is a projection onto $Q(\ell^p)$.

If $x\in E$, then $Qx=x$, and therefore
\[
 Px=Sx+x-Sx=x.
\]
Thus $P|_E=I_E$.  On the other hand, if $g\in G$, then
$S^*g=g$, and hence
\[
 P^*g
 =
 S^*g+Q^*g-Q^*S^*g
 =
 g+Q^*g-Q^*g
 =
 g.
\]
Thus $P^*|_G=I_G$.

Finally,
\[
 \|P\|
 \leq \|S\|+\|Q\|+\|S\|\,\|Q\|
 \leq 2K_p+K_p^2,
\]
while, since $P$ is a projection onto $Q(\ell^p)$,
\[
 \rank P=\rank Q
 \leq f_p\bigl(n+f_p(n)\bigr).
\]
Hence the conclusion holds, for example, with
\[
 C_p:=2K_p+K_p^2,
 \qquad
 h_p(n):=f_p\bigl(n+f_p(n)\bigr).
\]
\end{proof}

Lemma~\ref{lem:quantitative-primal-dual-projection} immediately yields
the simultaneous absorption statement needed below.  Compare
\cite[Lemma~4.2]{Lusky2003} for the corresponding two-sided absorption
result without the quantitative rank estimate required here.

\begin{lem}[Simultaneous finite-rank projections]\label{lem:simultaneous-projection}
For each $1\leq p<\infty$, let $C_p\geq1$ and $h_p:\N\to\N$ be as in Lemma~\ref{lem:quantitative-primal-dual-projection}.  If $R_1,\ldots,R_s\in B(\ell^p)$ are finite-rank operators satisfying
\[
 \sum_{j=1}^s\rank(R_j)\leq d,
\]
then there is a finite-rank projection $Q\in B(\ell^p)$ such that
\[
 QR_j=R_jQ=R_j
 \qquad(1\leq j\leq s),
\]
\[
 \|Q\|\leq C_p,
 \qquad
 \rank Q\leq h_p(d).
\]
\end{lem}

\begin{proof}
If $d=0$, then every $R_j$ is zero, and we may take $Q=0$.
Assume $d\geq1$, and set
\[
 E=\sum_{j=1}^s\ran R_j\subseteq\ell^p,
 \qquad
 G=\sum_{j=1}^s\ran R_j^*
       \subseteq(\ell^p)^*.
\]
Since $\rank R_j^*=\rank R_j$,
\[
 \dim E\leq d,
 \qquad
 \dim G\leq d.
\]

Applying Lemma~\ref{lem:quantitative-primal-dual-projection} with
$n=d$, we obtain a finite-rank projection $Q\in B(\ell^p)$ such that
\[
 Q|_E=I_E,\qquad Q^*|_G=I_G,
 \qquad
 \|Q\|\leq C_p,\qquad
 \rank Q\leq h_p(d).
\]

Since $\ran R_j\subseteq E$,
\[
 QR_j=R_j.
\]
Since $\ran R_j^*\subseteq G$,
\[
 (R_jQ)^*=Q^*R_j^*=R_j^*,
\]
and hence $R_jQ=R_j$.  Therefore
\[
 QR_j=R_jQ=R_j
 \qquad(1\leq j\leq s),
\]
with the asserted uniform norm and rank bounds.
\end{proof}

\subsection{Projection local units}

For $r\geq0$ set
\[
 \kappa_r:=\sup_{x\in X}|B(x,r)|<\infty.
\]

For a Banach algebra $D$ and a finite subset $\mathcal F\subseteq D$,
an element $e\in D$ is a \emph{two-sided local unit} for $\mathcal F$
if
\[
 ea=a=ae
 \qquad(a\in\mathcal F).
\]

\begin{prop}\label{prop:common-projection}
For every $1\leq p<\infty$ there is a constant $C_p\geq1$ such that every nonempty finite subset of $U\C^p[X]$ has a two-sided local unit which is a projection of norm at most $C_p$.
\end{prop}

\begin{proof}
Let $T^{(1)},\ldots,T^{(m)}\in U\C^p[X]$.  Choose $r\geq0$ and $N\in\N$ such that
\[
 \Prop(T^{(j)})\leq r,
 \qquad
 \rank(T^{(j)}_{xy})\leq N
\]
for all $j,x,y$.  For each $x\in X$, consider the finite family of operators on $\ell^p$
\[
 \mathcal R_x
 :=\{T^{(j)}_{xy}:y\in B(x,r),\ 1\leq j\leq m\}
 \cup
 \{T^{(j)}_{yx}:y\in B(x,r),\ 1\leq j\leq m\}.
\]
There are at most $2m\kappa_r$ members, each of rank at most $N$, so
\[
 \sum_{R\in\mathcal R_x}\rank R\leq2m\kappa_rN
\]
uniformly in $x$.  By Lemma~\ref{lem:simultaneous-projection}, for each $x$ there is a projection $Q_x\in B(\ell^p)$ such that
\[
 Q_xR=RQ_x=R\qquad(R\in\mathcal R_x),
\]
\[
 \|Q_x\|\leq C_p,
 \qquad
 \rank Q_x\leq h_p(2m\kappa_rN).
\]

Define
\[
 Q=\bigoplus_{x\in X}Q_x
 \in B(\ell^p(X,\ell^p)).
\]
Then $Q^2=Q$, $\|Q\|\leq C_p$, and $Q$ has propagation zero.  Its nonzero matrix entries are the $Q_x$, whose ranks are uniformly bounded.  Hence $Q\in U\C^p[X]$.

For every $j,x,y$,
\[
 (QT^{(j)})_{xy}=Q_xT^{(j)}_{xy}=T^{(j)}_{xy}.
\]
Also
\[
 (T^{(j)}Q)_{xy}=T^{(j)}_{xy}Q_y.
\]
If $T^{(j)}_{xy}\neq0$, then $x\in B(y,r)$, so $T^{(j)}_{xy}$ occurs in the second family defining $\mathcal R_y$.  Hence $T^{(j)}_{xy}Q_y=T^{(j)}_{xy}$.  Therefore
\[
 QT^{(j)}=T^{(j)}=T^{(j)}Q
 \qquad(1\leq j\leq m).
\]
\end{proof}

\begin{thm}\label{thm:UBp-bai-projections}
Let $X$ be a metric space with bounded geometry and $1\leq p<\infty$.  Then $UB^p(X)$ has a bounded two-sided approximate identity consisting of projections.
\end{thm}

\begin{proof}
Fix a nonempty finite set $F=\{S^{(1)},\ldots,S^{(m)}\}\subseteq UB^p(X)$ and $\varepsilon>0$.  Let $C_p$ be as in Proposition~\ref{prop:common-projection}.  Since $U\C^p[X]$ is dense in $UB^p(X)$, choose $T^{(j)}\in U\C^p[X]$ such that
\[
 \|S^{(j)}-T^{(j)}\|<\frac{\varepsilon}{C_p+1}.
\]
By Proposition~\ref{prop:common-projection}, there is a projection $Q\in U\C^p[X]$ with $\|Q\|\leq C_p$ and
\[
 QT^{(j)}=T^{(j)}=T^{(j)}Q
\]
for every $j$.  Hence
\begin{align*}
 \|QS^{(j)}-S^{(j)}\|
 &\leq \|Q(S^{(j)}-T^{(j)})\|+\|T^{(j)}-S^{(j)}\|<\varepsilon,\\
 \|S^{(j)}Q-S^{(j)}\|
 &\leq \|(S^{(j)}-T^{(j)})Q\|+\|T^{(j)}-S^{(j)}\|<\varepsilon.
\end{align*}

For each pair $(F,\varepsilon)$, where
$F\subseteq UB^p(X)$ is nonempty and finite and $\varepsilon>0$,
choose such a projection $Q_{F,\varepsilon}$.  Order these pairs by
\[
(F,\varepsilon)\preceq(F',\varepsilon')
\quad\Longleftrightarrow\quad
F\subseteq F'
\ \text{ and }\ 
\varepsilon'\leq\varepsilon.
\]
Then $(Q_{F,\varepsilon})$ is a net of projections with
$\|Q_{F,\varepsilon}\|\leq C_p$.  Given $S\in UB^p(X)$ and
$\delta>0$, for every sufficiently large $(F,\varepsilon)$ we have
$S\in F$ and $\varepsilon<\delta$, and hence
\[
\|Q_{F,\varepsilon}S-S\|<\delta,
\qquad
\|SQ_{F,\varepsilon}-S\|<\delta.
\]
Thus $(Q_{F,\varepsilon})$ is a bounded two-sided approximate identity
consisting of projections.
\end{proof}

Although a bounded approximate identity would suffice for the
ideal-induction results, the construction above yields projections;
for $p=2$ they may be chosen orthogonal and contractive, while for
$p\neq2$ their norms are uniformly bounded by a constant depending
only on $p$.

\section{Geometric ideals in \texorpdfstring{$\ell^p$}{l-p} uniform algebras}

We now apply Theorem \ref{thm:geometric-preservation} to the Morita equivalence constructed in \cite{Chung2025}.  Let $X$ be a metric space with bounded geometry and let $1\leq p<\infty$.  
Write $1/p+1/q=1$, with $q=\infty$ when $p=1$.

Let $\C_u[X]$ denote the algebra of finite-propagation $X\times X$ complex matrices with uniformly bounded entries.  
Bounded geometry implies that every such matrix defines a bounded operator on $\ell^p(X)$.  The $\ell^p$ uniform Roe algebra is
\[
 B_u^p(X):=\overline{\C_u[X]}^{\|\cdot\|_{B(\ell^p(X))}}.
\]
As in Section~6, $U\C^p[X]$ denotes the algebra of finite-propagation operator-valued matrices on $\ell^p(X,\ell^p)$ with uniformly bounded matrix-entry ranks, and $UB^p(X)$ is its operator-norm closure.

Let $E_{\mathrm{alg}}^{>}$ be the space of finite-propagation $X\times X$ matrices with uniformly bounded entries in $\ell^p$, regarded as operators
\[
 \ell^p(X)\longrightarrow\ell^p(X,\ell^p).
\]
Let $E_{\mathrm{alg}}^{<}$ be the corresponding space of finite-propagation matrices with uniformly bounded entries in $\ell^q$, regarded as operators
\[
 \ell^p(X,\ell^p)\longrightarrow\ell^p(X).
\]
Let $E^{>}$ and $E^{<}$ be their operator-norm completions. By \cite[Theorem~3.1]{Chung2025},
\[
 ({}_{\Bup}E^{<}_{\UBp},{}_{\UBp}E^{>}_{\Bup})
\]
is a Morita equivalence.  The $B_u^p(X)$-valued pairing is matrix multiplication
\[
 \langle e^{<},e^{>}\rangle_{\Bup}=e^{<}e^{>},
\]
and the $UB^p(X)$-valued pairing is
\[
 {}_{\UBp}\langle e^{>},e^{<}\rangle=e^{>}e^{<},
\]
where the latter has operator-valued matrix entries.

We call a closed ideal $I\triangleleft B_u^p(X)$ \emph{geometric} if
\[
 \overline{I\cap \C_u[X]}=I,
\]
and a closed ideal $J\triangleleft UB^p(X)$ \emph{geometric} if
\[
 \overline{J\cap U\C^p[X]}=J.
\]
Thus these are precisely the $\C_u[X]$-geometric and $U\C^p[X]$-geometric ideals, respectively, in the sense of Definition~\ref{def:core-geometric}.

The algebraic Morita data form the dense subcontext required in Section~3.

\begin{prop}\label{prop:roe-dense-context}
The quadruple
\[
 \bigl(U\C^p[X],\C_u[X],E_{\mathrm{alg}}^{<},E_{\mathrm{alg}}^{>}\bigr)
\]
is a compatible dense subcontext of the Morita equivalence between $UB^p(X)$ and $B_u^p(X)$.
\end{prop}

\begin{proof}
Density is part of the definitions and of the construction in
\cite{Chung2025}.  We verify the algebraic conditions.

Products of finite-propagation matrices have finite propagation; bounded geometry gives a uniform bound on the number of terms
occurring in each matrix-product entry.  Together with the uniform
bounds on the scalar, vector, functional and finite-rank operator
entries, this gives all four module-stability conditions in
Definition~\ref{def:dense-subcontext}.

It remains to check the pairings.  For
$e^{<}\in E_{\mathrm{alg}}^{<}$ and
$e^{>}\in E_{\mathrm{alg}}^{>}$, the $(y,z)$-entry of
$e^{<}e^{>}$ is
\[
 \sum_w
 \langle e^{<}_{yw},e^{>}_{wz}\rangle.
\]
It has finite propagation and uniformly bounded scalar entries, so
\[
 \langle
 E_{\mathrm{alg}}^{<},
 E_{\mathrm{alg}}^{>}
 \rangle_{\Bup}
 \subseteq
 \C_u[X].
\]

Similarly, the $(x,y)$-entry of $e^{>}e^{<}$ is
\[
 \sum_z e^{>}_{xz}\otimes e^{<}_{zy},
\]
where $v\otimes\varphi$ denotes the rank-one operator
$w\mapsto\varphi(w)v$ on $\ell^p$.  Bounded geometry gives a
uniform bound on the number of nonzero summands, so these entries
have uniformly bounded rank; the propagation is again finite.
Thus
\[
 {}_{\UBp}\langle
 E_{\mathrm{alg}}^{>},
 E_{\mathrm{alg}}^{<}
 \rangle
 \subseteq
 U\C^p[X].
\]
Hence all the conditions of
Definition~\ref{def:dense-subcontext} are satisfied.
\end{proof}

\begin{cor}\label{cor:roe-full-ideal-lattice}
Let $X$ be a metric space with bounded geometry and
$1\leq p<\infty$.  The Morita equivalence of \cite{Chung2025}
induces an isomorphism of complete lattices
\[
 \Ind_{B_u^p(X)}^{UB^p(X)}:
 \mathcal I(B_u^p(X))
 \xrightarrow{\cong}
 \mathcal I(UB^p(X)).
\]
\end{cor}

\begin{proof}
The algebra $B_u^p(X)$ is unital, while
Theorem~\ref{thm:UBp-bai-projections} gives a bounded approximate
identity for $UB^p(X)$.  The assertion is therefore
Theorem~\ref{thm:ideal-lattice}.
\end{proof}

\begin{thm}\label{thm:roe-geometric-correspondence}
Let $X$ be a metric space with bounded geometry and $1\leq p<\infty$.  The Morita equivalence of \cite{Chung2025} induces an order isomorphism
\[
 \mathcal G\bigl(B_u^p(X);\C_u[X]\bigr)
 \xrightarrow{\cong}
 \mathcal G\bigl(UB^p(X);U\C^p[X]\bigr).
\]
Equivalently, a closed ideal $I\triangleleft B_u^p(X)$ is geometric if and only if the induced ideal $\Ind_{B_u^p(X)}^{UB^p(X)}(I)$ is geometric.
\end{thm}

\begin{proof}
This is Theorem~\ref{thm:geometric-preservation} applied to Proposition~\ref{prop:roe-dense-context}.
\end{proof}

\subsection{Coarse-ideal parametrization}
\label{subsec:coarse-ideal-parametrization}

Recall that an \emph{ideal of the bounded coarse structure}
$\mathcal E_d$ is a coarse structure
$\mathcal J\subseteq\mathcal E_d$ such that
\[
 E\circ F,\ F\circ E\in\mathcal J
 \qquad
 (E\in\mathcal E_d,\ F\in\mathcal J).
\]
In particular, $\mathcal J$ is closed under taking subsets, inverses
and finite unions.  Let $\mathcal J(\mathcal E_d)$ denote the lattice
of such coarse ideals.

Since the zero ideal of an operator algebra does not correspond to a
coarse structure, we adjoin a formal bottom element and write
\[
 \widehat{\mathcal J}(\mathcal E_d)
 :=
 \mathcal J(\mathcal E_d)\sqcup\{\mathbf 0\},
\]
where $\mathbf 0$ is declared to be below every element of
$\mathcal J(\mathcal E_d)$.

For $E\subseteq X\times X$, write
\[
 r(E):=\{x\in X:(x,y)\in E\text{ for some }y\in X\},
\]
and
\[
 s(E):=\{y\in X:(x,y)\in E\text{ for some }x\in X\}.
\]

For a coarse ideal
$\mathcal J\in\mathcal J(\mathcal E_d)$, set
\[
 \C_u[X,\mathcal J]
 :=
 \{T\in\C_u[X]:\supp(T)\in\mathcal J\},
\]
\[
 B_u^p(X,\mathcal J)
 :=
 \overline{\C_u[X,\mathcal J]}^{\,\|\cdot\|},
\]
and
\[
 U\C^p[X,\mathcal J]
 :=
 \{T\in U\C^p[X]:\supp(T)\in\mathcal J\},
\]
\[
 UB^p(X,\mathcal J)
 :=
 \overline{U\C^p[X,\mathcal J]}^{\,\|\cdot\|}.
\]
For the formal bottom element, set
\[
 B_u^p(X,\mathbf 0)
 :=
 UB^p(X,\mathbf 0)
 :=
 \{0\}.
\]

\begin{prop}\label{prop:explicit-coarse-induction}
For every
$\mathcal J\in\widehat{\mathcal J}(\mathcal E_d)$,
\[
 \Ind_{B_u^p(X)}^{UB^p(X)}
 \bigl(B_u^p(X,\mathcal J)\bigr)
 =
 UB^p(X,\mathcal J).
\]
\end{prop}

\begin{proof}
The assertion is trivial for $\mathcal J=\mathbf 0$, so assume that
$\mathcal J$ is a coarse ideal.

We first prove the inclusion from left to right.  If
\[
 e^{>}\in E_{\mathrm{alg}}^{>},\qquad
 b\in \C_u[X,\mathcal J],\qquad
 e^{<}\in E_{\mathrm{alg}}^{<},
\]
then there are controlled sets $F,G\in\mathcal E_d$ such that
\[
 \supp(e^{>})\subseteq F,
 \qquad
 \supp(e^{<})\subseteq G.
\]
Hence
\[
 \supp(e^{>}be^{<})
 \subseteq
 F\circ\supp(b)\circ G.
\]
Since $\mathcal J$ is an ideal of $\mathcal E_d$, the right-hand
side belongs to $\mathcal J$.  Thus
\[
 e^{>}be^{<}\in U\C^p[X,\mathcal J].
\]
Using the density of the algebraic Morita data and of
$\C_u[X,\mathcal J]$ in $B_u^p(X,\mathcal J)$, together with
Proposition~\ref{prop:induction-formulas}, gives
\[
 \Ind_{B_u^p(X)}^{UB^p(X)}
 \bigl(B_u^p(X,\mathcal J)\bigr)
 \subseteq UB^p(X,\mathcal J).
\]

For the reverse inclusion, let
$T\in U\C^p[X,\mathcal J]$.  Since $X$ has bounded geometry, the bounded coarse structure
$\mathcal E_d$ is uniformly locally finite.  By
\cite[Lemma~2.5(ii)]{ChungDu2026}, the controlled relation
$\supp(T)$ admits a finite disjoint decomposition
\[
 \supp(T)=E_1\sqcup\cdots\sqcup E_m
\]
into partial translations, that is, relations $E_l$ on which both
coordinate projections are injective.  Let $T^{(l)}$ be the matrix obtained by restricting
$T$ to $E_l$.  Then
\[
 T=\sum_{l=1}^m T^{(l)}.
\]

Choose $N$ such that
\[
 \rank(T_{xy})\leq N
 \qquad(x,y\in X).
\]
For each $(x,y)\in E_l$ with $T_{xy}\neq0$, the finite-dimensional
Auerbach lemma allows us to choose an Auerbach system
\[
 \bigl(v_{xy}^{(k)},\psi_{xy}^{(k)}\bigr)_
 {k=1}^{\rank(T_{xy})}
\]
for the finite-dimensional space $\ran(T_{xy})$.  Thus $\|v_{xy}^{(k)}\|
 =
 \|\psi_{xy}^{(k)}\|
 =
 1$,
and the system is biorthogonal:
\[
 \psi_{xy}^{(k)}
 \bigl(v_{xy}^{(j)}\bigr)
 =
 \delta_{kj}.
\]
The point of the Auerbach normalization is that these bounds are
independent of $(x,y)$.  This uniformity is needed below in order
that the vectors and functionals occurring in the rank-one
decompositions define elements of
$E_{\mathrm{alg}}^{>}$ and $E_{\mathrm{alg}}^{<}$, respectively.

Set
\[
 \varphi_{xy}^{(k)}
 :=
 \psi_{xy}^{(k)}\circ T_{xy}.
\]
After padding with zeros up to $N$ terms, we obtain
\[
 T_{xy}
 =
 \sum_{k=1}^N
 v_{xy}^{(k)}\otimes\varphi_{xy}^{(k)},
\]
with
\[
 \|v_{xy}^{(k)}\|\leq1,
 \qquad
 \|\varphi_{xy}^{(k)}\|
 \leq\|T_{xy}\|
 \leq\|T\|.
\]
Indeed, for $z\in\ell^p$,
\[
 T_{xy}z
 =
 \sum_{k=1}^{\rank(T_{xy})}
 \psi_{xy}^{(k)}(T_{xy}z)\,
 v_{xy}^{(k)},
\]
which is precisely the displayed rank-one decomposition.  Notice
that the composition
$\psi_{xy}^{(k)}\circ T_{xy}$ is already a bounded functional on
all of $\ell^p$.

Since both coordinate projections are injective on $E_l$, define
diagonal matrices
$V_{l,k}\in E_{\mathrm{alg}}^{>}$ and
$W_{l,k}\in E_{\mathrm{alg}}^{<}$ by
\[
 (V_{l,k})_{xx}
 =
 \begin{cases}
 v_{xy}^{(k)},&(x,y)\in E_l,\\
 0,&\text{otherwise},
 \end{cases}
\]
and
\[
 (W_{l,k})_{yy}
 =
 \begin{cases}
 \varphi_{xy}^{(k)},&(x,y)\in E_l,\\
 0,&\text{otherwise}.
 \end{cases}
\]
These matrices have propagation zero and uniformly bounded entries.
Moreover,
\[
 \chi_{E_l}\in\C_u[X,\mathcal J],
\]
because $E_l\subseteq\supp(T)\in\mathcal J$.  A direct matrix
calculation gives
\[
 T^{(l)}
 =
 \sum_{k=1}^N
 V_{l,k}\chi_{E_l}W_{l,k}.
\]
Each summand belongs to the algebraic span defining
\[
 \Ind_{B_u^p(X)}^{UB^p(X)}
 \bigl(B_u^p(X,\mathcal J)\bigr).
\]
Hence $T$ belongs to the induced ideal.  Taking norm closures gives
the reverse inclusion.
\end{proof}

\begin{prop}\label{prop:coarse-ideal-projection-bai}
Let
$\mathcal J\in\widehat{\mathcal J}(\mathcal E_d)$.
Then $UB^p(X,\mathcal J)$ has a bounded two-sided approximate
identity consisting of projections.  The projections may be chosen
with norms bounded by the same constant $C_p$ as in
Proposition~\ref{prop:common-projection}.
\end{prop}

\begin{proof}
The zero ideal is immediate, so assume that $\mathcal J$ is a coarse
ideal.

Let
\[
 T^{(1)},\ldots,T^{(m)}
 \in U\C^p[X,\mathcal J].
\]
Put
\[
 E:=\bigcup_{j=1}^m\supp(T^{(j)})\in\mathcal J
\]
and
\[
 Y:=r(E)\cup s(E).
\]
Since $\mathcal J$ is closed under inverses, compositions, finite
unions and taking subsets, $\Delta_Y\in\mathcal J$;
indeed,
\[
 \Delta_{r(E)}\subseteq E\circ E^{-1},
 \qquad
 \Delta_{s(E)}\subseteq E^{-1}\circ E.
\]

Repeat the construction in
Proposition~\ref{prop:common-projection}, choosing the local
projection $Q_x$ there for $x\in Y$ and setting $Q_x=0$ for
$x\notin Y$.  The resulting diagonal projection
\[
 Q=\bigoplus_{x\in X}Q_x
\]
still satisfies
\[
 QT^{(j)}=T^{(j)}=T^{(j)}Q
 \qquad(1\leq j\leq m),
\]
and
\[
 \|Q\|\leq C_p.
\]
Moreover,
\[
 \supp(Q)\subseteq\Delta_Y\in\mathcal J,
\]
and the ranks of the matrix entries of $Q$ are uniformly bounded.
Thus $Q\in U\C^p[X,\mathcal J]$.

The approximation argument in the proof of
Theorem~\ref{thm:UBp-bai-projections} now applies verbatim to the
dense subalgebra $U\C^p[X,\mathcal J]$ of
$UB^p(X,\mathcal J)$.
\end{proof}

\begin{cor}\label{cor:coarse-ideals-morita}
For every
$\mathcal J\in\widehat{\mathcal J}(\mathcal E_d)$, the
corresponding ideals
\[
 B_u^p(X,\mathcal J)
 \qquad\text{and}\qquad
 UB^p(X,\mathcal J)
\]
are Morita equivalent through the restricted bimodules.
\end{cor}

\begin{proof}
By Proposition~\ref{prop:explicit-coarse-induction}, the two ideals
correspond under Morita induction;
Proposition~\ref{prop:coarse-ideal-projection-bai} gives a bounded
approximate identity for $UB^p(X,\mathcal J)$.  Apply
Theorem~\ref{thm:corresponding-ideals-morita}.
\end{proof}

\begin{rem}\label{rem:coarse-ideal-projection-bai}
The two corresponding ideals in
Corollary~\ref{cor:coarse-ideals-morita} both admit bounded
approximate identities of projections.  On the scalar side one can
take a contractive approximate identity of diagonal projections:
for
$\mathcal J\in\mathcal J(\mathcal E_d)$, direct $\mathcal J$ by
inclusion and put
\[
 e_E:=\chi_{r(E)}
 \qquad(E\in\mathcal J).
\]
Since
\[
 \Delta_{r(E)}\subseteq E\circ E^{-1}\in\mathcal J,
\]
we have $e_E\in\C_u[X,\mathcal J]$.  If
$T\in\C_u[X,\mathcal J]$, then once $E$ contains
$\supp(T)\cup\supp(T)^{-1}$,
\[
 e_ET=T=Te_E.
\]
Thus $(e_E)$ is a contractive two-sided approximate identity for
$B_u^p(X,\mathcal J)$.  The $UB^p(X,\mathcal J)$-side is
Proposition~\ref{prop:coarse-ideal-projection-bai}.
\end{rem}

The companion paper \cite[Proposition~4.19]{ChungDu2026} proves that the nonzero
geometric ideals of $B_u^p(X)$ are parametrized by the ideals of the
bounded coarse structure $\mathcal E_d$.  Adjoining the formal
bottom element $\mathbf 0$ and sending it to the zero ideal extends
this to the zero-augmented coarse-ideal lattice introduced above.
We therefore obtain the following.

\begin{cor}\label{cor:UB-classification}
Let $X$ be a metric space with bounded geometry and
$1\leq p<\infty$.  Then the geometric ideals of $UB^p(X)$ are
precisely
\[
 UB^p(X,\mathcal J),
 \qquad
 \mathcal J\in\widehat{\mathcal J}(\mathcal E_d).
\]
Moreover,
\[
 \mathcal J
 \longmapsto
 UB^p(X,\mathcal J)
\]
is a lattice isomorphism
\[
 \widehat{\mathcal J}(\mathcal E_d)
 \xrightarrow{\cong}
 \mathcal G\bigl(UB^p(X);U\C^p[X]\bigr).
\]
In particular, the geometric ideal lattice of $UB^p(X)$ is
independent of $p$ via the common coarse-ideal lattice.
\end{cor}

\begin{proof}
By \cite[Theorem~4.21(i)]{ChungDu2026}, the nonzero geometric
ideals of $B_u^p(X)$ are precisely
\[
 B_u^p(X,\mathcal J),
 \qquad
 \mathcal J\in\mathcal J(\mathcal E_d),
\]
and this parametrization is an order isomorphism.  After adjoining
the formal bottom element $\mathbf 0$ and setting
\[
 B_u^p(X,\mathbf 0)=\{0\},
\]
it extends to a lattice isomorphism
\[
 \widehat{\mathcal J}(\mathcal E_d)
 \xrightarrow{\cong}
 \mathcal G\bigl(B_u^p(X);\C_u[X]\bigr).
\]
The conclusion now follows from
Theorem~\ref{thm:roe-geometric-correspondence} and
Proposition~\ref{prop:explicit-coarse-induction}.
\end{proof}

\subsection{Essential-support and geometric-core blocks}
\label{subsec:essential-support-blocks}

The coarse-ideal parametrization above describes the geometric
ideals themselves.  We now use the essential-support theory of
\cite[Definition~4.23 and Proposition~4.24]{ChungDu2026} to obtain a finer description of the full ideal
lattice.  The geometric core of an arbitrary ideal determines a
block indexed by a coarse ideal, and Morita induction transports
these blocks together with their smallest and largest elements.

For a nonzero closed ideal $I\triangleleft B_u^p(X)$, recall from
\cite[Definition~4.23 and Proposition~4.14]{ChungDu2026} its essential coarse support
\[
 \mathcal J(I)
 :=
 \{\supp_\varepsilon(T):
     T\in I,\ \varepsilon>0\}
 \in\mathcal J(\mathcal E_d).
\]
For the zero ideal, set
\[
 \mathcal J(\{0\}):=\mathbf 0.
\]
For
$\mathcal J\in\widehat{\mathcal J}(\mathcal E_d)$, define the
corresponding geometric-core block by
\[
 \mathfrak I_u[\mathcal J]
 :=
 \left\{
 I\triangleleft B_u^p(X):
 \operatorname{gcore}_{\C_u[X]}(I)
 =
 B_u^p(X,\mathcal J)
 \right\}.
\]

\begin{prop}\label{prop:gcore-essential-support}
For every closed ideal $I\triangleleft B_u^p(X)$,
\[
 \operatorname{gcore}_{\C_u[X]}(I)
 =
 B_u^p\bigl(X,\mathcal J(I)\bigr).
\]
Consequently, for every
$\mathcal J\in\widehat{\mathcal J}(\mathcal E_d)$,
\[
 \mathfrak I_u[\mathcal J]
 =
 \{I\triangleleft B_u^p(X):
   \mathcal J(I)=\mathcal J\}.
\]
Thus the geometric-core blocks defined here are exactly the
essential-support blocks of \cite[Definition~4.23]{ChungDu2026}.
\end{prop}

\begin{proof}
The assertion is immediate for $I=\{0\}$, so suppose that
$I\neq\{0\}$ and put
\[
 G:=\operatorname{gcore}_{\C_u[X]}(I).
\]
Since $I\neq\{0\}$, choose $0\neq T_0\in I$ and
$\varepsilon_0>0$ with
$E_0=\supp_{\varepsilon_0}(T_0)\neq\varnothing$.
By \cite[Proposition~4.10(ii)]{ChungDu2026},
$0\neq\chi_{E_0}\in I\cap\C_u[X]\subseteq G$, so $G\neq\{0\}$.
Since $G\subseteq I$,
\[
 \mathcal J(G)\subseteq\mathcal J(I).
\]

For the reverse inclusion, let
\[
 E=\supp_\varepsilon(T),
 \qquad
 T\in I,\quad \varepsilon>0.
\]
By \cite[Proposition~4.10(ii)]{ChungDu2026},
$\chi_E$ belongs to $\C_u[X]$ and to the closed ideal generated by
$T$.  Since $T\in I$, we therefore have
\[
 \chi_E
 =
 \chi_{\supp_\varepsilon(T)}
 \in I\cap\C_u[X]
 \subseteq G.
\]
Therefore
\[
 E=\supp_1(\chi_E)\in\mathcal J(G).
\]
Hence
\[
 \mathcal J(G)=\mathcal J(I).
\]

By Proposition~\ref{prop:geometric-core-preservation}, $G$ is
geometric.  The classification
\cite[Proposition~4.19]{ChungDu2026} therefore gives
\[
 G
 =
 B_u^p\bigl(X,\mathcal J(G)\bigr)
 =
 B_u^p\bigl(X,\mathcal J(I)\bigr).
\]
The final assertion follows immediately.
\end{proof}

\begin{cor}[Geometric-core blocks]
\label{cor:roe-geometric-core-blocks}
For
$\mathcal J\in\widehat{\mathcal J}(\mathcal E_d)$, let
\[
 \mathfrak I_U[\mathcal J]
 :=
 \left\{
 J\triangleleft UB^p(X):
 \operatorname{gcore}_{U\C^p[X]}(J)
 =
 UB^p(X,\mathcal J)
 \right\}.
\]
For a coarse ideal
$\mathcal J\in\mathcal J(\mathcal E_d)$, let
\[
 \widetilde I(\mathcal J)
 :=
 \left\{
 T\in B_u^p(X):
 \supp_\varepsilon(T)\in\mathcal J
 \text{ for every }\varepsilon>0
 \right\},
\]
and put
\[
 \widetilde I(\mathbf 0):=\{0\}.
\]
Writing
\[
 \Phi:=\Ind_{B_u^p(X)}^{UB^p(X)},
\]
define
\[
 \widetilde{UB}^{\,p}(X,\mathcal J)
 :=
 \Phi\bigl(\widetilde I(\mathcal J)\bigr).
\]

Then
\[
 \mathfrak I_u[\mathcal J]
 =
 \bigl[
 B_u^p(X,\mathcal J),
 \widetilde I(\mathcal J)
 \bigr]
\]
and
\[
 \mathfrak I_U[\mathcal J]
 =
 \bigl[
 UB^p(X,\mathcal J),
 \widetilde{UB}^{\,p}(X,\mathcal J)
 \bigr],
\]
where the brackets denote intervals in the respective ideal
lattices.  In particular, both blocks are complete lattices with
the order inherited from the ambient ideal lattices.  Moreover,
Morita induction restricts to an isomorphism of complete lattices
\[
 \mathfrak I_u[\mathcal J]
 \xrightarrow{\cong}
 \mathfrak I_U[\mathcal J].
\]
In particular, the full ideal lattices on the two sides decompose
into corresponding blocks indexed by
$\widehat{\mathcal J}(\mathcal E_d)$, with corresponding smallest
and largest elements.
\end{cor}

\begin{proof}
By Proposition~\ref{prop:gcore-essential-support},
$\mathfrak I_u[\mathcal J]$ is precisely the essential-support block
associated with $\mathcal J$ in \cite[Definition~4.23]{ChungDu2026}.
For
$\mathcal J\neq\mathbf 0$,
\cite[Propositions~4.19 and 4.24]{ChungDu2026} shows that its smallest and
largest elements are, respectively,
\[
 B_u^p(X,\mathcal J)
 \qquad\text{and}\qquad
 \widetilde I(\mathcal J).
\]
The assertion for $\mathcal J=\mathbf 0$ is immediate.  Hence
\[
 \mathfrak I_u[\mathcal J]
 =
 \bigl[
 B_u^p(X,\mathcal J),
 \widetilde I(\mathcal J)
 \bigr].
\]

By Proposition~\ref{prop:geometric-core-preservation} and
Proposition~\ref{prop:explicit-coarse-induction}, $\Phi$ maps
$\mathfrak I_u[\mathcal J]$ onto
$\mathfrak I_U[\mathcal J]$.  Since $\Phi$ is an order isomorphism
by Corollary~\ref{cor:roe-full-ideal-lattice}, it carries the two
endpoints of the preceding interval onto
\[
 UB^p(X,\mathcal J)
 \qquad\text{and}\qquad
 \widetilde{UB}^{\,p}(X,\mathcal J),
\]
respectively.  Hence
\[
 \mathfrak I_U[\mathcal J]
 =
 \bigl[
 UB^p(X,\mathcal J),
 \widetilde{UB}^{\,p}(X,\mathcal J)
 \bigr].
\]
Both intervals are complete lattices, and $\Phi$ restricts to an
isomorphism of complete lattices between them.
\end{proof}

\begin{cor}\label{cor:UB-all-ideals-geometric}
Let $X$ be a metric space with bounded geometry and
$1\leq p<\infty$.  Suppose either
\begin{enumerate}
\item[(i)] $p=1$, or
\item[(ii)] $1<p<\infty$ and $X$ has property~A.
\end{enumerate}
Then every closed two-sided ideal of $UB^p(X)$ is geometric.
Consequently,
\[
 \mathcal I(UB^p(X))
 \cong
 \widehat{\mathcal J}(\mathcal E_d).
\]
Moreover, every geometric-core block is a singleton: for every
$\mathcal J\in\widehat{\mathcal J}(\mathcal E_d)$,
\[
 \mathfrak I_U[\mathcal J]
 =
 \{UB^p(X,\mathcal J)\},
\]
and hence
\[
 \widetilde{UB}^{\,p}(X,\mathcal J)
 =
 UB^p(X,\mathcal J).
\]
\end{cor}

\begin{proof}
If $p=1$, \cite[Corollary~6.22]{ChungDu2026} shows that every closed
ideal of $B_u^p(X)$ is geometric.  If $1<p<\infty$ and $X$ has
property~A, the same conclusion follows from
\cite[Corollary~6.22]{ChungDu2026}.

By Corollary~\ref{cor:roe-full-ideal-lattice}, every closed ideal
of $UB^p(X)$ is the Morita induction of a unique ideal of
$B_u^p(X)$.  Theorem~\ref{thm:geometric-preservation} therefore
shows that every closed ideal of $UB^p(X)$ is geometric.  The
lattice identification follows from
Corollary~\ref{cor:UB-classification}.

Finally, if
$J\in\mathfrak I_U[\mathcal J]$, then $J$ is geometric and hence
\[
 J
 =
 \operatorname{gcore}_{U\C^p[X]}(J)
 =
 UB^p(X,\mathcal J).
\]
Thus $\mathfrak I_U[\mathcal J]$ is a singleton.  Since
$\widetilde{UB}^{\,p}(X,\mathcal J)$ is the largest element of this
block by Corollary~\ref{cor:roe-geometric-core-blocks}, the final
equality follows.
\end{proof}

\begin{cor}[Property A and the minimal block]
\label{cor:UB2-property-A-block}
Let $X$ be a metric space with bounded geometry, and let
$\mathcal E_{\min}$ denote the smallest coarse ideal on $X$,
consisting of the finite subsets of $X\times X$.  The following are
equivalent:
\begin{enumerate}
\item $X$ has property~A;
\item every closed two-sided ideal of $UB^2(X)$ is geometric;
\item the block $\mathfrak I_U[\mathcal E_{\min}]$ is a singleton;
\item $\widetilde{UB}^{\,2}(X,\mathcal E_{\min})
 =
 UB^2(X,\mathcal E_{\min})$.
\end{enumerate}
\end{cor}

\begin{proof}
If $X$ has property~A, then
Corollary~\ref{cor:UB-all-ideals-geometric} shows that every closed
two-sided ideal of $UB^2(X)$ is geometric.  Thus (i) implies (ii).

Assume (ii).  If $J\in\mathfrak I_U[\mathcal E_{\min}]$, then $J$ is geometric and hence
\[
 J
 =
 \operatorname{gcore}_{U\C^2[X]}(J)
 =
 UB^2(X,\mathcal E_{\min}).
\]
Thus the block $\mathfrak I_U[\mathcal E_{\min}]$ is a singleton,
so (ii) implies (iii).

By Corollary~\ref{cor:roe-geometric-core-blocks},
\[
 \mathfrak I_U[\mathcal E_{\min}]
 =
 \bigl[
 UB^2(X,\mathcal E_{\min}),
 \widetilde{UB}^{\,2}(X,\mathcal E_{\min})
 \bigr].
\]
Hence (iii) and (iv) are equivalent.

Finally, assume (iv).  Applying the inverse Morita induction map and
using Corollary~\ref{cor:roe-geometric-core-blocks} gives
\[
 \widetilde I(\mathcal E_{\min})
 =
 B_u^2(X,\mathcal E_{\min}).
\]
By \cite[Corollary~6.13 and proof of Proposition~6.23]{ChungDu2026}, the
right-hand side is the compact-operator ideal and the left-hand side
is the ghost ideal of $C_u^*(X)$.  Hence every ghost in
$C_u^*(X)$ is compact.  By
\cite[Theorem~1.3]{RoeWillett2014}, $X$ has property~A.
Thus (iv) implies (i).
\end{proof}

The quotient theorem also yields a companion statement.

\begin{cor}\label{cor:roe-quotients}
Let $I\triangleleft B_u^p(X)$ be a closed ideal, and let
\[
 J=\Ind_{B_u^p(X)}^{UB^p(X)}(I).
\]
Then
\[
 UB^p(X)/J
 \quad\text{and}\quad
 B_u^p(X)/I
\]
are Morita equivalent.  In particular, for every
$\mathcal J\in\widehat{\mathcal J}(\mathcal E_d)$, the quotient
algebras
\[
 UB^p(X)/UB^p(X,\mathcal J)
 \qquad\text{and}\qquad
 B_u^p(X)/B_u^p(X,\mathcal J)
\]
are Morita equivalent.
\end{cor}

\begin{proof}
The first assertion is
Theorem~\ref{thm:quotient-morita} applied to the Morita equivalence
from \cite{Chung2025}.  The final assertion follows from
Proposition~\ref{prop:explicit-coarse-induction}.
\end{proof}

\section*{Acknowledgments}

The first author is supported by the Research Start-up Fund of Jilin University (grant number 419080524100).

\bibliographystyle{amsplain}
\bibliography{morita_induction_ideals}

\end{document}